\documentclass{article}
\usepackage[utf8]{inputenc}
\usepackage[T1]{fontenc}
\usepackage{CJKutf8}
\usepackage[utf8]{inputenc}
\usepackage{lmodern}
\usepackage{euscript}
\usepackage{amsmath}
\usepackage{amssymb}
\usepackage{tabularx}
\usepackage{hyperref}
\usepackage{multicol}
\usepackage{eufrak}
\usepackage[bottom]{footmisc}
\usepackage{color}
\usepackage{titling}
\usepackage{rotating}
\usepackage{lipsum}
\usepackage{titlesec}
\usepackage{graphicx}
\usepackage{pdfpages}
\usepackage{amsthm}
\usepackage{mathabx}
\usepackage[margin=1in]{geometry}

\newcommand{\ben}{\begin{enumerate}}
\newcommand{\een}{\end{enumerate}}
\newcommand{\beq}{\begin{quote}}
\newcommand{\enq}{\end{quote}}

\newtheorem{thm}{Theorem}[section]
\newtheorem{cor}{Corollary}[section]
\newtheorem{lemma}[thm]{Lemma}
\newtheorem{rmk}[thm]{Remark}
\newtheorem{definition}[thm]{Definition}
\newtheorem{prop}[thm]{Proposition}

\usepackage[english]{babel}
\newcommand{\bthm}{\begin{thm}}
\newcommand{\ethm}{\end{thm}}
\newcommand{\bprf}{\begin{proof}}
\newcommand{\eprf}{\end{proof}}

\usepackage{tcolorbox}
\usepackage{tikz-cd}
\usepackage[all,cmtip]{xy}

\usepackage{lmodern}
\usepackage{euscript}
\usepackage{amsmath}
\usepackage{amssymb}
\usepackage{tabularx}
\usepackage{tikz-cd}

\newtheorem{example}[thm]{Example}

\begin{document}

\title{Bott-Samelson-Demazure-Hansen Varieties 
for Projective Homogeneous Varieties with Nonreduced Stabilizers }
\author{Siqing Zhang}
\date{}

\maketitle

\begin{abstract}
Over a field of positive characteristic, a semisimple algebraic group $G$ may have some nonreduced parabolic subgroup $P$. In this paper, we study the Schubert and Bott-Samelson-Demazure-Hansen (BSDH) varieties of $G/P$, with $P$ nonreduced, when the base field is perfect. 
    It is shown that in general the Schubert and BSDH varieties of such a $G/P$ are not normal, and  the projection of the BSDH variety onto the Schubert variety has nonreduced fibers at closed points.
    When the base field is finite, the generalized convolution morphisms between BSDH varieties (as in \cite{dCHL}) are also studied. It is shown that the decomposition theorem holds for such morphisms, and the pushforward of intersection complexes by such morphisms are Frobenius semisimple. 
\end{abstract}

\section{Introduction}

Let $G$ be a connected split semisimple linear algebraic group over a perfect field $k$. A parabolic subgroup $P$ is a sub-group scheme of $G$ that contains a Borel subgroup. 
When char$k>0$, such a parabolic $P$ can be nonreduced. 
In \cite[Theorem 14]{Wenzel}, Wenzel gives a classification of the nonreduced parabolics over an algebraically closed field $k$ of char$(k)>3$. 

It is natural to consider the quotient $G/P$ where $P$ is nonreduced. The quotient $G/P$ is a projective $G$-homogeneous variety with a nonreduced stabilizer $P$. \\

In contrast to the classical flag varieties, which enjoy many nice properties, the quotient $G/P$, with $P$ nonreduced, is exotic, as demonstrated by the following examples:

Haboush and Lauritzen, whose works are over algebraically closed $k$, showed that some $G/P$, with nonreduced $P$, 
\begin{enumerate}
\item is not Frobenius split \cite[Theorem 5.2]{Lauthesis}; 
\item is not $\mathcal{D}$-affine
\cite[\S4.4]{Lau97}, i.e., it has a $\mathcal{D}$-module, which is quasi-coherent as an $\mathcal{O}$-module, with nonvanishing higher cohomology;
\item violates the Kodaira and Kempf vanishing theorems \cite[Ex.\;6.3.1-2]{Lauthesis}; 
\item admits ample line bundles with negative Euler characteristic \cite[Ex.\;6.3.3]{Lauthesis}, \item does not admit a flat lift to $ \mathbb{Z}$ \cite[56]{HL}.
\end{enumerate}

In \cite[Cor.\;2.2]{Totaro}, Totaro shows that the affine cone over certain $G/P$, with $P$ nonreduced, has an isolated terminal singularity that is not Cohen-Macaulay. This is the first discovery of such singularity.  

These properties make $G/P$, with $P$ nonreduced, a curious object to study.

In this paper, we define and study the  Bott-Samelson-Demazure-Hansen (BSDH) varieties for $G/P$, with $P$ nonreduced. To our knowledge, such study is new even for $k=\bar{k}$. A slogan for these newly constructed BSDH varieties is that "geometry is wild, and topology is nice". Here is what we mean:

Given a positive integer $r$ and an $r$-tuple $w_{\bullet}=(w_1,...,w_r)$ of elements in the Weyl group, we define a BSDH variety $X_P(w_{\bullet})$ which coincides with the classical BSDH variety if $P$ is reduced. Just as in the classical case, the BSDH variety admits canonical projections to $X_P(w_i)$ for each $i=1,...,r$. 
Given $Q$ a parabolic containing $P$, we can define some morphisms from BSDH varieties for $P$ to some BSDH varieties for $Q$. These morphisms are called the generalized convolution morphisms, and they are generalizations of both the canonical projections mentioned above and the canonical morphism $G/P\to G/Q$. 

"Geometry is wild": in Section \ref{geo} we show that the newly constructed BSDH variety $X_P(w_{\bullet})$, with $P$ nonreduced,
\begin{enumerate}
\item
is in general not an iterated fibration with fibers isomorphic to $X_P(w_i)$'s, instead, Examples \ref{Nonreduced Fibers of $p_1$}, \ref{sl52} and \ref{nonreducedfiber} show that the fibers of the first projection $p_1: X_P(w_{\bullet})\to X_P(w_1)$ in general are nonreduced. Furthermore, Ex.\;\ref{lastprojection2} shows that the fibers of the last projection $X_P(w_1,...,w_r)\to X_P(w_r)$ in general are also nonreduced; 
\item
is in general not normal. We give an example of a non-normal BSDH variety $X_P(s_1,s_2)$ in Ex.\;\ref{nonnormalbsdh}. 
In Ex.\;\ref{Nonnormalschubert},
for each $n\ge 2$, we give examples of $n$ dimensional non-normal Schubert varieties.
\end{enumerate}

"Topology is nice": 
in Section \ref{top} we show that over a finite or algebraically closed field $k$,
\begin{enumerate}
\item 
The decomposition theorem package holds for any generalized convolution morphism (defined in Def.\;\ref{generalizedconvolution}); 
\item 
The intersection complex $\mathcal{IC}_{X_P(w_{\bullet})}$, and the derived pushforward of it by a generalized convolution morphism, are semisimple, Frobenius semisimple, very pure of weight zero, even and Tate (definitions in \ref{gooddef}). 
\end{enumerate}
When $P$ is reduced, these were shown for the classical and some infinite dimensional versions of $X_P(w_{\bullet})$ in \cite{BGS}, \cite{BY}, \cite{AR}, and \cite{dCHL}. Here we take the finite dimensional part of their results and use the fact that the newly constructed BSDH varieties are universally homeomorphic to the classical BSDH varieties. \\

Here is the plan for this paper:

Firstly, we fix the notation.
In the second section, we study the Schubert varieties, the Richardson varieties, and the integral Chow ring of $G/P$, with $P$ nonreduced, over a perfect field. 
In the third section, we define the BSDH varieties $X_P(w_1,...,w_r)=X_P(w_{\bullet})$ for $G/P$ with $P$ nonreduced, and discuss some other possible ways to define them.
In the fourth section, we show the wild geometric properties of the newly constructed BSDH varieties listed above. 
In the fifth section, we show the nice topological properties of $X_P(w_{\bullet})$, with $P$ nonreduced, over a finite or algebraically closed field, as stated above. 

\subsection{Notation}\label{not}
Let $k$ be a perfect field with characteristic $p>0$.

All the schemes in this paper are separated and of finite type.

$G$ is a connected split semisimple linear algebraic group over $k$. Fix a split maximal torus $T$ in $G$ and $B$ a Borel subgroup containing $T$. Let $W$ be the Weyl group.

$P$ denotes a not necessarily reduced parabolic subgroup of $G$ that contains $B$. The maximal reduced subscheme of $P$ is denoted by $P_{red}$.

The set of roots $R$ is inside the character lattice $X(T)$ of $T$. For each root $\alpha\in R$, $-\alpha\in R\subset X(T)$ is the opposite root, and there is a root homomorphism $x_{\alpha}:\mathbb{G}_a\to G$. Let the subgroup $U(\alpha)$ be the image of $x_{\alpha}$. 

Let $\Delta\subset R$ be the subset of simple roots such that the corresponding set of the positive roots consists of the roots of $B$.

For every subset $I\subset R$, define $U(I)$ to be the subscheme of $G$ that is the image of the product morphism $(\prod_{\alpha\in I} x_{\alpha})(\mathbb{G}_a^{\# I})\subset G$. Note that $U(I)$ may not be a subgroup of $G$.

The unipotent radical $R_u(Q)$ of a reduced parabolic $Q$ is $U(I)$ for a unique subset $I\subset R$. In this case we also write $Q=P_I$. The opposite unipotent $R_u^-(Q)$ is $U(-I)$.  

\textbf{Note}: our notation $P_I$ differs from the standard textbooks, e.g. \cite{Springer}, \cite{Jantzen}, and \cite{MilneGrp} : in those books our $P_I$ is their $P_{\Delta\setminus \pm I}$. We opt for the lighter notation because we will be using the symbols $P_I$ and $U(I)$ a lot. 

Let $\mathbb{G}_{a,n}:= \text{Spec}(k[Y]/(Y^{p^n})$ be the additive infinitesimal group schemes, and $U(\alpha,n):= Image(x_{\alpha}(\mathbb{G}_{a,n}))$.

Define $\mathbb{G}_{a,\infty}:= \mathbb{G}_{a}$ and $U(\alpha,\infty):=U(\alpha)$.

Let $L_I=T\cdot U(R\setminus \pm I)$ be the Levi factor of $P_I$. Let $W^I\subset W$ be the Weyl group for $(L_I,T)$. Let $W_I\subset W$ be the sub-poset consisting of the longest representatives for the classes in the double coset $W^I\backslash W/ W^I$. 

For a subset $I\subset R$ and a function $\vec{n}:I\to \mathbb{N}_{>0}\cup\{\infty\}$,
$\vec{n}(\alpha)=n_{\alpha}$, define the subscheme of $G$
\[U(I,\vec{n}):= \prod_{\alpha\in I} x_{\alpha} (\mathbb{G}_{a,n_{\alpha}})=\prod_{\alpha \in I} U(\alpha, n_{\alpha}).\]

For $w\in W_I$, define the subscheme of $G$
\[U(w(I,\vec{n})):= \prod_{\alpha\in I} x_{w(\alpha)}(\mathbb{G}_{a,n_{\alpha}})= \prod_{\alpha\in I} U(w(\alpha),n_{\alpha}).\]

Recall a reduced scheme $X$ is \textit{paved by affine spaces} if there exists a sequence of closed subschemes $\emptyset\subset X_0\subset X_1...\subset X_n=X$ so that each $X_i\setminus X_{i-1}$ is isomorphic to a disjoint union of affine spaces $\mathbb{A}^{n_i}$ for some $n_i\in \mathbb{N}$. 
When $P$ is reduced, an affine paving of $G/P$ is given by the Bruhat decomposition and the Bruhat order:
\begin{equation}
\label{01}
    G/P=\coprod_{w\in W/W^I} B wP/P.
\end{equation}
We call ${}^{B}\!X_P'(w)=B wP/P$ (resp. ${}^{B}\!X_P(w)=\overline{B w P/P}$) the Schubert cell (resp. Schubert varieties) of $G/P$ corresponding to $w$.

When $P$ is reduced, there is another decomposition of $G/P$ as 
\begin{equation}
\label{02}
    G/P=\coprod_{w\in W_I}PwP/P.
\end{equation}
We denote $X_P'(w)=PwP/P$ and $X_P(w)=\overline{PwP/P}$.

Let $l$ be a prime number so that $l \ne char(k)=p$. When the field $k$ is finite, let $D_m^b(X,\overline{\mathbb{Q}_l})$ be the bounded mixed constructible derived category with the middle perversity $t$-structure as in \cite{BBD}. We denote the derived direct image $Rf_*$ just as $f_*$.\\

About the name of $G/P$ with $P$ nonreduced: $G/P$ has been originally called the varieties of unseparated flag (vufs) by Haboush, Lauritzen, and Wenzel, e.g. in \cite{HL}, and the projective pseudo-homogeneous spaces in \cite{Srimathy}. 
However, $G/P$ is a separated scheme and the action of $G$ on $G/P$ is transitive. 
To avoid further confusion, we just write them as $G/P$ and emphasize that $P$ is nonreduced when needed. 

 \subsection*{Acknowledgements}
I would like to thank my advisor, Mark de Cataldo, for many useful conversations. I would like to thank Michel Brion for very helpful comments on this paper. I learned about the existence of nonreduced parabolics from a course given by him in the 2019 graduate summer school on the Geometry and Modular Representation Theory of Algebraic Groups, held in the Simons Center. I would like to thank Thomas Haines for useful email exchanges. I would also like to thank many valuable comments from the anonymous referees.

\section{Basic Structures}\label{basic}

In this section, we first show some basic structure of the nonreduced parabolics in Prop.\;\ref{propP}. 
We then describe the integral Chow ring of $G/P$ with $P$ nonreduced in Prop.\;\ref{Chowring}. 
These results are stated and proved in \cite{Wenzel} and \cite{Lau97} when $k=\bar{k}$. 
Their proofs remain valid when $k$ is perfect.
We then describe a local product structure of the big cells of $G/P$ in terms of Schubert and Richardson varieties in Prop.\;\ref{opposite}.

\begin{prop} 
\label{propP}
With the settings in \ref{not},
\begin{enumerate}
\item
For every parabolic $P$, reduced or not, in $G$, $P_{red}$ is a parabolic subgroup of $G$, the unipotent radical of $P_{red}$ is $R_u(P_{red})=U(I)$ for a unique subset $I\subset R$, i.e., $P_{red}=P_I$. Further, 
$P=P_I\cdot (P\cap U(-I))$, i.e., $P$ is the image of the multiplication morphism over $k$, $P_I\times (P\cap U(-I))\to G$, where $P\cap U(-I)$ is the scheme-theoretic intersection.
\item
\[P\cap U(-I)=U(J,\vec{n})\cong \prod_{\beta\in J} Spec(k[T]/(T^{p^{n_{\beta}}})),\]
for a uniquely determined subset $J\subset -I$ and uniquely determined $n_{\beta}<\infty$ for each $\beta\in J$. 
\end{enumerate}
\end{prop}
\begin{proof}
When $k$ is algebraically closed and $p>3$, both items above were proved in \cite[Th.\;4, Th.\;10(i)]{Wenzel}. Our proof is similar in spirit.

Item (2) follows from item (1) the same way as in the argument from Lemma 5 to Theorem 10 (i) in \cite{Wenzel}. In fact, the argument there over algebraically closed fields holds verbatim over any field. Therefore it is enough to show item (1).

Because $k$ is perfect, $P_{red}$ is a reduced parabolic subgroup scheme of $G$ \cite[Cor.\;1.39]{MilneGrp}, therefore $P_{red}=P_I$ for a unique subset $I\subset R$ \cite[Th.\;21.91]{MilneGrp}. 

Clearly, $P$ contains $P_I\cdot (P\cap U(-I))$ as a subscheme. If $P$ is a subscheme of $P_I\cdot U(-I)$, then $P$ is a subscheme of $P_I\cdot (P\cap U(-I))$. 
Hence suffices to show $P$ is a subscheme of $P_I\cdot U(-I)$. 

By \cite[162]{Jantzen}, $P_I\cdot U(-I)$ is an open and dense subscheme of $G$, with a complement closed subscheme $Z\subset G$. 

If $P$ is not a subscheme of $P_I.U(-I)$, then the scheme $P\cap Z$ is not empty, thus $(P_I\cap Z_{red})_{red}=(P\cap Z)_{red}$ \cite[Sec.\;I-5.1.8]{EGA} is nonempty. Therefore $P_I\cap Z \ne \emptyset$, which contradicts the definition of $Z$.

\end{proof}

\begin{rmk}
Let us keep the notation from Prop.\;\ref{propP}.
When $k$ has characteristic $p>3$ or when $G$ is simply laced, it is proved in \cite[Th.\;10.(iii)]{Wenzel} that  
\begin{equation}
\label{func}
n_{\beta}=\min\{n_{\delta}|\; \delta\in \Delta\cap J,\; \langle \beta^{\vee}, \delta\rangle  \ne 0\},
\end{equation}
where $\beta^{\vee}\in X^{\vee}(T)$ is the dual root of $\beta$ in the cocharacter lattice. Let 
$$\mathcal{S}:= \{\vec{n}: \Delta^-\to \mathbb{N}_{>0}\cup \{\infty\}|\; \vec{n} \text{ is not constant }\},$$
where $\Delta^-$ is the set of negative simple roots.
The equality (\ref{func}) means that for a fixed Borel subgroup, there is a bijection $f: \mathcal{S}\xrightarrow{\sim} \mathcal{P}$, where $\mathcal{P}$ is the set of nonreduced parabolics containing $B$. \\
When $p\le 3$ it is shown in \cite[Rmk.\;15]{Wenzel} that in general $f$ is only an injection.
\end{rmk}

\begin{rmk}
When $k$ is imperfect, there exists an algebraic subgroup scheme $H\subset \mathbb{G}_a^n$ over $k$ whose reduced part $H_{red}$ is not a group \cite[Sec.\;$VI_A$.1.3.2]{SGA}. This leads to a natural question of which we do not know the answer: does there exist a parabolic subgroup $P$ of a linear algebraic group $G$ so that $P_{red}$ is not a group anymore? 
\end{rmk}

From the general properties of quotients by subgroup schemes of a smooth algebraic group \cite[Thm.\;2.7.2.(2)]{brion2017some}, we see that $G/P$ is a smooth scheme and a homogeneous space. The natural map $\pi: G/P_{red}\to G/P$ is purely inseparable and the fiber of $\pi$ over the identity flag $F_{id}\in G/P$ is the infinitesimal scheme $P/P_{red}$ \cite[Rmk.\;7.16.(a)]{MilneGrp}. 

The following proposition establishes the Bruhat decomposition for $G/P$ when $P$ is nonreduced. This can be seen as a manifestation of the slogan "topology is nice" that will be discussed in detail in Section 5. 

We call the images under $\pi$ of the Schubert cells (resp. varieties) of $G/P_{red}$ as the Schubert cells (resp. varieties) of $G/P$, and denote them as \[{}^B\!X_P'(w):=\pi({}^B\!X_{P_{red}}'(w)),\;\;{}^B\!X_P(w):=\pi({}^B\!X_{P_{red}}(w)).\]

\begin{prop}
Let $P$ be a nonreduced parabolic of $G$ and let $\pi: G/P_{red}\to G/P$ be the natural quotient map. Then
\begin{enumerate}
\item The Schubert cells form an affine paving of $G/P$;
\item The integral Chow groups $Ch(G/P)$ and $Ch(G/P_{red})$ are isomorphic as abstract abelian groups.
\end{enumerate}
\end{prop}
\begin{proof}
Both items are proved in \cite[Sec.\;2.3]{Lau97} in the context that $k$ is algebraically closed and $p>3$.

 Over perfect fields, both items are proved in \cite[Lem.\;5.1-2]{Srimathy} following Lauritzen's idea.
\end{proof}

The next proposition compares the ring structures of $Ch(G/P)$ and $Ch(G/P_{red})$.
It is stated and proved in \cite[6]{Lau97} over an algebraically closed field $k$. The same proof applies also in the case when $k$ is perfect.

Let $P$, $I$ and $J$ be as in Proposition \ref{propP}, we define 
$ d_w:=\sum_{\alpha\in w^{-1}(I)\cap J}n_{\alpha}$. Let $w_0\in W/W^I$ be the longest element.

\begin{prop}
\label{Chowring}
Let $P\subset G$ be a nonreduced parabolic. The pullback $\pi^*: Ch(G/P)\to Ch(G/P_{red})$ between the integral Chow rings is a ring embedding and the cokernel of $\pi^*$ is a $p^N$-torsion abelian group. Moreover, 
\[\pi_*[{}^B\!X_{P_I}(w)]=p^{d_w}[{}^B\!X_P(w)],\;\;\; \pi^*[{}^B\!X_P(w)]=p^{d_{w_0}-d_w}[{}^B\!X_{P_I}(w)].\]
\end{prop}

For any $w\in W/W^I$, let $Y_P'(w)$ be the $U(-I)$-orbit of the $T$-fixed point $wP/P$. Let $Y_P(w)$ be the closure of $Y_P'(w)$. When $P$ is reduced, $Y_P'(w)$ (resp. $Y_P(w)$) is called the opposite Schubert cell (resp. variety). It is easy to see that $Y_P(w)$ is the scheme theoretic image of $Y_{P_{red}}(w)$ under the natural morphism $G/P_{red}\to G/P$. 
The following Proposition \ref{opposite} can be proved in the same way as \cite[Lemma A4.(b)]{kl1979}.

\begin{prop}
\label{opposite}
Let $v\in W/W^I$ be such that $w>v$. We then have
\[{}^B\!X_{P}(w)\cap vY_{P}'(id)\cong {}^B\!X_P'(v)\times (Y'_{P}(v)\cap {}^B\!X_{P}(w)). \]
\end{prop}

Proposition \ref{opposite} reflects another piece of the basic structure of a flag variety:  
Assume now that $P$ is reduced. The open Richardson varieties in $G/P$
are defined as 
\[Z'_{P}(w,v):={}^B\!X_{P}'(w)\cap Y_{P}'(v).\]

For each $T$-fixed point $vP/P\in G/P$, the big cell around it is $vU(-I)P/P=vY_{P}'(id)$. 
Intersecting with a Schubert variety, we have that the big cell of a Schubert variety at a $T$-fixed point is isomorphic to the direct product of an affine space and a slice through it, which lies between an open Richardson variety and its closure \cite[Prop.\;1.3.5]{Brion}. Proposition \ref{opposite} says that we have a similar decomposition for $G/P$ with $P$ nonreduced.

\section{BSDH Varieties for $G/P$ with $P$ Nonreduced: Definitions}

We assume the set up in \ref{not} and \ref{propP}. 
In particular, the base field $k$ is perfect.

Let $w_{\bullet}$ be a sequence of elements $w_1,...,w_r\in W_I$.
We introduce the protagonist of this paper: the Bott-Samelson-Demazure-Hansen (BSDH) variety $X_P(w_{\bullet})$ when the parabolic $P$ is nonreduced.

When $P$ is reduced, the BSDH varieties are well known and thoroughly studied in the literature. 
For the definition, let us refer to \cite[\S II.13]{Jantzen} and \cite[\S 4.1]{dCHL}. 
In particular, a BSDH variety $X_P(w_{\bullet})$ is an integral closed subscheme of $(G/P)^r$. 
Let $k'$ be a finite field extension of $k$, and let $g_{0}=id$ be the identity element of $G(k')$. 
A closed point $(g_1P/P,..., g_rP/P)$ of $(G/P)^r$ is in $X_P(w_1,...,w_r)(k')$ 
if and only if $g_{i-1}^{-1}g_i\in \overline{Pw_iP}(k')$ for $i=1,...,r$.
When $r=1$, the BSDH variety $X_P(w_1)$ is the same as the subvariety $X_P(w_1)=\overline{Pw_1P/P}$ of $G/P$ as in (\ref{02}).

When $P$ is nonreduced, as we have seen in Sec. \ref{basic}, 
the Schubert varieties of $G/P$ are defined as the scheme theoretic image of the Schubert varieties  of 
$G/P_{red}$ under the natural universal homeomorphism $\pi:G/P_{red}\to G/P$. 
We define the BSDH varieties similarly:

\begin{definition}[BSDH Varieties]
Let $P$ be a nonreduced parabolic of $G$. 
For $w_{\bullet}=(w_1,...,w_r)\in W_I^r$, we define the BSDH variety $X_P(w_{\bullet})$ to be 
the scheme theoretic image of the classical BSDH variety $X_{P_{red}}(w_{\bullet})\subset (G/P)^r$ 
under the natural morphism $\pi^r:(G/P_{red})^r\to (G/P)^r$. 
\end{definition}

We immediately have the integrality of $X_P(w_{\bullet})$ and the characterization of its closed points as in the classical case:

\begin{prop}
\label{firstproperties}
Let $P$ be a not necessarily reduced parabolic.
\begin{enumerate}
\item 
The BSDH variety $X_P(w_{\bullet})$ is an integral closed subscheme of $(G/P)^r$ that is universally homeomorphic to $X_{P_{red}}(w_{\bullet})$;
\item
A closed point $(a_1,...,a_r)$ of $(G/P)^r$,
with residue field $k'$, is in $X_P(w_{\bullet})$
if and only if there exists $k'$-points $g_1,...,g_r$ of $G$ so that $a_i=g_iP/P$ and $g_{i-1}^{-1}g_i\in \overline{P_{red}w_iP_{red}}(k')=\overline{Pw_iP}(k')$, where
$i=1,...,r$ and $g_0:=id$ is the identity element of $G$.
\end{enumerate}
\end{prop}
\begin{proof}
(1) As $\pi^r$ is proper, the underlying space of $X_P(w_{\bullet})$ is the image of the underlying space of $X_{P_{red}}(w_{\bullet})$ under $\pi^r$. 
Therefore $X_{P_{red}}(w_{\bullet})$ is universally homeomorphic to $X_P(w_{\bullet})$ via $\pi^r$. 
Since $X_{P_{red}}(w_{\bullet})$ is irreducible, we have that $X_P(w_{\bullet})$ is also irreducible. 
By \cite[Lem.\;29.6.7]{stacks-project},
the scheme theoretic image of a reduced scheme is reduced. Hence $X_P(w_{\bullet})$ is also reduced, as $X_{P_{red}}(w_{\bullet})$ is.

(2) Let $(a_1,...,a_r)$ be a closed point of $X_P(w_{\bullet})$ with residue field $k'$. Let $(b_1,...,b_r)$ be the closed point of 
$X_{P_{red}}(w_{\bullet})$ over $(a_1,...,a_r)$. As $\pi^r$ is finite and purely inseparable, the residue field $k''$ of $(b_1,...,b_r)$ is a finite purely inseparable
extension of $k'$. 
Since $k$ is perfect, we have that $k'$ is also perfect. Thus $k''=k'$.
By the equality  $(\overline{Pw_iP})_{red}=\overline{P_{red}w_iP_{red}}$ of reduced closed subschemes of $G$, we have that $\overline{Pw_iP}$ and $\overline{P_{red}w_iP_{red}}$ have the same $K$-points, for any field extension $K\supset k$.
The rest follows from the characterization of closed points of $X_{P_{red}}(w_{\bullet})$.
\end{proof}

The Prop.\;\ref{chara} below gives us another way to define the BSDH varieties, no matter whether $P$ is reduced or not. 

\begin{definition}[BSDH Functor]
\label{BSDHfun}
Let $P$ be a not necessarily reduced parabolic. We define a $k$-functor $\mathcal{X}_P(w_{\bullet})$ to be the subfunctor of $(G/P)^r$ that sends every $k$-algebra $A$ to the subset $\mathcal{X}_P'(w_{\bullet})(A)$ of $(G/P)^r(A)=((G/P)(A))^r$ that consists of the points $(a_1,...,a_r)$ such that there exist $g_0\in P_{red}(A)$, $g_1,...,g_r\in G(A)$ satisfying that $g_i P/P=a_i$ and $g_{i-1}^{-1} g_i\in \overline{(P_{red} w_i P_{red})}(A)$  for $i=1,...,r$.
\end{definition}

\begin{prop}
\label{chara}
Let $P$ be a not necessarily reduced parabolic. 
The BSDH variety $X_P(w_{\bullet})$ is the smallest closed subscheme of $(G/P)^r$ that contains $\mathcal{X}_P(w_{\bullet})$ as a subfunctor, i.e., the inclusion of $k$-functors $\mathcal{X}_P(w_{\bullet})\hookrightarrow (G/P)^r$ factors as $\mathcal{X}_P(w_{\bullet})\xrightarrow{i_1} X_P(w_{\bullet})\xrightarrow{i_2} (G/P)^r$, where $i_1$ is an inclusion of $k$-functors and $i_2$ is a closed immersion of schemes.
\end{prop}
\begin{proof}
When $P$ is reduced, this characterization of $X_P(w_{\bullet})$ is indeed one of the standard definitions of the BSDH varieties.
For example, let us refer to \cite[Sec.\;II.13]{Jantzen}.
There $X_{P_{red}}(w_{\bullet})$ is defined to be the big fppf sheaf associated with the image functor, under the quotient morphism 
$G^r\to (G/P_{red})^r$, of the closed subscheme $V(w_{\bullet})$ of $G^r$, whose functor of points sends a $k$-algebra $A$ to the set
\[V(w_{\bullet})(A):=\{(g_1,...,g_r)\in G^r(A)|\;\; g_{i-1}^{-1}g_i\in\overline{P_{red}w_iP_{red}}(A)\},\]
where $g_0$ is the identity element. 
By definition, the image functor of $V(w_{\bullet})$ is just $\mathcal{X}_{P_{red}}(w_{\bullet})$.
Since $V(w_{\bullet})$ is $P$-invariant, 
the big fppf sheaf $X_{P_{red}}(w_{\bullet})$ is indeed a closed subscheme of $(G/P_{red})^r$ by  \cite[Sec.\;I.5.21]{Jantzen}.

When $P$ is nonreduced, the scheme $X_P(w_{\bullet})$ is defined to be the scheme theoretic image of $X_{P_{red}}(w_{\bullet})$ under $\pi^r$, it hence suffices to show that $\mathcal{X}_P(w_{\bullet})$ is 
the image functor, i.e., $\mathcal{X}_P(w_{\bullet})(A)=\pi(A)\mathcal{X}_{P_{red}}(w_{\bullet})(A)$. 
To see this, take any $(a_1,...,a_r)\in \mathcal{X}_P(w_{\bullet})(A)$, take the corresponding $g_i$ in Def.\;\ref{BSDHfun}, then the $A$-point $(g_1P_{red}/P_{red},...,g_rP_{red}/P_{red})$ 
of $(G/P_{red})^r$ is in $\mathcal{X}_{P_{red}}(w_{\bullet})$ because of the conditions satisfied by the $g_i$'s. Since $\pi(A)(g_i P_{red}/P_{red})=g_i P/P$, we have that $(a_1,...,a_r)= \pi(A)(g_i P_{red}/P_{red})_i$. Therefore $\pi(A)(\mathcal{X}_{P_{red}}(w_{\bullet})(A))\supset \mathcal{X}_P(w_{\bullet})(A)$.
On the other hand, if $(b_1,...,b_r)\in \mathcal{X}_{P_{red}}(w_{\bullet})(A)$, then take the corresponding $g_i
$'s in Def.\;\ref{BSDHfun}. Again because $\pi(A)(g_i P_{red}/P_{red})=g_i P/P$, we have $\pi(A)(b_1,...,b_r)\in \mathcal{X}_P(w_{\bullet})(A)$.
\end{proof}

Some basic information about $k$-functors and big fppf sheaves are included in the Appendix: Sec.\ref{appendix}. 
There we show that the big fppf-sheaf associated with an image $k$-functor in general is not a scheme.

\begin{rmk}
\label{rmkflat1}
In the proof above, we mentioned that the scheme $X_{P_{red}}(w_{\bullet})$ is indeed the big fppf sheaf associated to 
the functor (presheaf) $\mathcal{X}_{P_{red}}(w_{\bullet})$. 
However, Rmk.\;\ref{notflatincidence} below entails that, when $P$ is nonreduced, in general the big fppf sheaf $\underline{\mathcal{X}_P}(w_{\bullet})$ associated to the functor $\mathcal{X}_P(w_{\bullet})$ is not a scheme.

Indeed, if $\underline{\mathcal{X}_P}(w_{\bullet})$ is a scheme, then by the universal property of fppf sheafification and Prop.\ref{chara}, we see that the inclusion of functors $\underline{\mathcal{X}_P}(w_{\bullet})\hookrightarrow X_P(w_{\bullet})$ gives rise to an isomorphism between schemes.
By \cite[Sec.\;I.5.4.4]{Jantzen}, we have that for every $k$-algebra $A$,
\[\underline{\mathcal{X}_P}(w_{\bullet})(A)=(G/P)^r(A)\cap \bigcup_B \pi^r(B)X_{P_{red}}(B),\]
where $B$ ranges over all fppf-$A$-algebras. 
However, from Rmk.\;\ref{notflatincidence} below, we see that the morphism 
$X_{P_{red}}(w_{\bullet})\to X_P(w_{\bullet})$
in general is not flat.
Therefore the $A$-points of $X_P(w_{\bullet})$
in general contain some points in $\pi^r(B)X_{P_{red}}(B)$ where $B$ is not a flat $A$-algebra. Therefore in general there is a strict inclusion of sets $X_P(w_{\bullet})(A)\supsetneq\underline{\mathcal{X}_P}(w_{\bullet})(A)$.
\end{rmk}

When $P$ is reduced, we can define the BSDH varieties using the relative positions between two flags. For the rest of the section, which is not used in other parts of this paper, we define a natural generalization of relative positions when $P$ is nonreduced, and show that they define in general nonreduced schemes. Therefore, we cannot define BSDH varieties using relative relations when $P$ is nonreduced, at least not using the natural definition of them below:

\begin{definition}[Relative Positions]
\label{relativeposition}
Let $P$ be a not necessarily reduced parabolic. For any $k$-algebra $A$, and $A$ points $a_1,...,a_r\in G/P(A)$, we define the relation $a_1\stackrel{ w_2}{\text{\textemdash}} a_2\stackrel{ w_3}{\text{\textemdash}}...\stackrel{ w_r}{\text{\textemdash}} a_r$ (resp. $a_1\stackrel{ \le w_2}{\text{\textemdash}} a_2\stackrel{\le w_3}{\text{\textemdash}}...\stackrel{\le w_r}{\text{\textemdash}} a_r$) if there are $g_1,..., g_r\in G(A)$ so that  $g_i P/P=a_i$ and  that for $i=2,...,r$, $g_{i-1}^{-1}g_i\in P_{red} w_i P_{red}(A)\subset G(A)$ (resp. $g_{i-1}^{-1}g_i\in \overline{P_{red} w_i P_{red}}(A)$).
\end{definition}

Note that the subscheme $X_{P_{red}}'(w)=P_{red}wP_{red}/P_{red}$ of $G/P_{red}$ can be defined as all the flags that are of relative position $w$ to the identity flag $P_{red}/P_{red}\in G/P_{red}$.  

 The following proposition shows that when $P$ is nonreduced, all the flags that are of relative position $w$ to the identity flag form a scheme that is in general nonreduced. 
 Therefore relative position is not very useful to define Schubert or BSDH varieties when $P$ is nonreduced.

\begin{prop}
For every $k$-algebra $A$, we have an equality of sets
\[\{a\in G/P(A)|\; P/P\stackrel{w}{\text{\textemdash}} a\}=P(A)\cdot wP/P.\]

Let $PwP/P$ be the scheme theoretic image of $P\times wP/P$ under the action morphism $G\times G/P\to G/P$. Then we have that $PwP/P$ in general is nonreduced.
\end{prop}
\begin{proof}
Every element in $P(A)\cdot wP/P$ has the form $\epsilon pwP/P$ where $\epsilon\in U(J,\vec{n})(A)$, $p\in P_I(A)$. 
Because $P/P= \epsilon P/P$, in the notation in Def. \ref{relativeposition}, 
we can choose $g_1=\epsilon$, $g_2=\epsilon pw$, so that $g_1P/P=P/P$, $g_2P/P=a$ and $g_0^{-1}g_1\in P_I w P_I(A)$.

For an example of nonreduced $PwP/P$. Let $G$ and $P$ be as in Example \ref{Nonreduced Fibers of $p_1$}. Let $w=s_{\alpha}$. Then $PwP/P$ is isomorphic to 
$U(-\alpha-\beta,1)\cdot U(-\alpha)P/P$, which is not reduced.
\end{proof}

\section{Wild Geometry}\label{geo}
We keep the notation in \ref{not} and \ref{propP}. Recall for $w_1,...,w_r\in W_I$, the geometric Demazure product $w_{\star}:=w_I\star ...\star w_r\in W_I$ is defined so that the image of the last projection $pr_r: X_{P_I}(w_{\bullet})\to G/{P_I}$ is $X_{P_I}(w_{\star})$. For a detailed discussion of the $\star$ operation, let us refer to \cite[Sec.\;4.2-3]{dCHL}.\\

When the parabolic $P$ is reduced, two morphisms among BSDH varieties are especially useful: the first projection $p_1:X_P(w_1,...,w_r)\to X_P(w_1)$ and the last projection $p_r: X_P(w_1,...,w_r)\to X_P(w_{\star})$. By studying the first projection, we see that $X_P(w_1,...,w_r)$ is a Zariski locally trivial fibration with fiber isomorphic to $X_P(w_2,...,w_r)$. 
In particular, if $P=B$ and $s_{\star}=s_1\cdot...\cdot s_r$ is a reduced expression in terms of simple reflections, then we have that $X_B(s_{\bullet})$ is an iterated $\mathbb{P}^1$-fibration, and the last projection gives a resolution of singularities of the normal Schubert variety $X_B(s_{\star})$. 

The picture is different when $P$ is nonreduced. 
In Ex.\;\ref{Nonreduced Fibers of $p_1$}, Ex.\;\ref{nonreducedfiber} and Ex.\;\ref{lastprojection2}, we show that the fibers of the first and last projections are in general nonreduced. 
In Ex.\;\ref{Nonnormalschubert} and Ex.\;\ref{nonnormalbsdh} below, we give examples of non-normal Schubert and BSDH varieties. 

For a not necessarily reduced parabolic $P$, let $I\subset R$ so that $P_{red}=P_I$. 

\begin{prop}[Fibers of the First Projection $p_1$]
\label{Fiber of First Projection}
Given any closed point $gP/P$ in $X_P(w_1)$, the first projection $p_1: X_P(w_1,...,w_r)\to X_P(w_1)$ has fibers isomorphic to 
\[(gP\cap\overline{P_{red}w_1P_{red}})\cdot X_P(w_2,...,w_r),\]
i.e., the scheme theoretic image of $(gP\cap\overline{P_{red}w_1P_{red}})\times X_P(w_2,...,w_r)$ under the morphism $G\times (G/P)^{r-1}\to (G/P)^{r-1}: (g, a_2,...,a_r)\mapsto (ga_2,...,ga_r)$.

In particular, the largest reduced subscheme of a fiber of $p_1$ is always isomorphic to $X_P(w_2,...,w_r)$.
\end{prop}

\begin{proof}
Let $k'$ be the residue field of the closed point $gP/P$.
Let $\mathcal{X}_P(w_{\bullet})$ be the image $k$-functor of $X_{P_{red}}(w_{\bullet})$ under the morphism $\pi^r:(G/P_{red})^r\to (G/P)^r$ as defined in Def.\;\ref{BSDHfun}.
By Prop.\;\ref{chara}, we have that $X_P(w_{\bullet})$ is the smallest closed subscheme of $(G/P)^r$ such that the inclusion of $k$-functors $\mathcal{X}_P(w_{\bullet})\hookrightarrow (G/P)^r$ factors as $\mathcal{X}_P(w_{\bullet})\hookrightarrow X_P(w_{\bullet})\hookrightarrow (G/P)^r$, where the first arrow is an inclusion of functors and the second arrow is the closed immersion of schemes.
Similarly, let $\mathcal{I}:=(gP\cap\overline{P_{red}w_1P_{red}})\cdot \mathcal{X}_P(w_2,...,w_r)$ be the image $k'$-functor under the multiplication. We have that $(gP\cap\overline{P_{red}w_1P_{red}})\cdot X_P(w_2,...,w_r)$ 
is the smallest closed subscheme of $(G/P)^{r-1}$ that contains $\mathcal{I}$ as a subfunctor.
Therefore, it suffices to show that for any $k'$-algebra $A$, the preimage of $gP/P\in \mathcal{X}_P(w_1)(A)$ under $p_1(A):\mathcal{X}_P(w_1,...,w_r)(A)\to \mathcal{X}_P(w_1)(A)$ is $gP/P\times\mathcal{I}(A)\subset \mathcal{X}_P(w_1,...,w_r)(A)$. 
Below we only consider $A$-points of functors.

By Def.\;\ref{BSDHfun}, the fiber consists of the points $(a_1,...,a_r)$ such that there exists $g_1,...,g_r\in G(A)$ with $g_1P/P=gP/P$, $g_1\in \overline{P_{red}w_1P_{red}}(A)$, and that $g_{i-1}^{-1}g_i\in \overline{P_{red}w_iP_{red}}(A)$ for $i=2,...,r$. 
The first two conditions give that an $A$-point $g_1\in G(A)$ can be a representative of the first factor of the fiber of $p_1$ over $gP/P$ if and only if $g_1\in gP\cap \overline{P_{red}w_1P_{red}}$. The third condition gives that, for example, an element $g_2\in G(A)$ can be a representative of the second factor of $p_1^{-1}(gP/P)$ if and only if $g_2\in (gP\cap \overline{P_{red}w_1P_{red}})\cdot X_P(w_2)$, and the representatives of the $i$-th factor for $i\ge 3$ are determined so iteratively. Therefore we have the first statement.

The second statement follows from the fact that $gP\cdot X_P(w_2,...,w_r)=gU(J,\vec{n})\cdot X_P(w_2,...,w_r)$ is an infinitesimal thickening of $X_P(w_2,...,w_r)$.
\end{proof}

As $X_P(w_2,...,w_r)$ is invariant under the left multiplication by $P_{red}$, Prop.\;\ref{Fiber of First Projection} tells us that to understand the fiber of the first projection $p_1$ over $gP/P$ is the same as to understand $gP/P_{red}\cap X_{P_{red}}(w_1)$, which in turn is the fiber of $\pi:X_{P_{red}}(w_1)\to X_P(w_1)$ over $gP/P$. 
In Propositions \ref{finer} and \ref{finer2} below, we give finer descriptions of the fibers of $p_1$ in terms of roots and Weyl groups.

Recall that we have the decomposition $X_P(w_1)=\coprod_{v\le w, v\in W_I} X_P'(v)$.
In Proposition \ref{finer}, we describe the fibers of $p_1$ over a point in the largest part $X_P'(w_1)$. 
In Proposition \ref{finer2}, we describe the fiber of $p_1$ over the smallest part $X_P(id)=P/P$ when $P_{red}=B$.

\begin{prop}[Finer Description of $p_1^{-1}(gP/P)$, I]
\label{finer}
If $gP/P$ is in $X_P'(w_1)=P_{red}w_1P/P$, then the fiber of the first projection $p_1:X_P(w_{\bullet})\to X_P(w_1)$ at $gP/P$ is isomorphic to 
\[(U(J,\vec{n})\cap w_1^{-1}P_{red}w_1)\cdot X_P(w_2,...,w_r).\]
\end{prop}

In Prop.\;\ref{finer2} below,
let $P_{red}=B$, and $w_1=s_{\alpha_1}...s_{\alpha_n}$ be a reduced expression of $w_1$ in terms of reflections $s_{\alpha_i}$ that exchanges the simple root $\alpha_i$ with its negative $-\alpha_i$. Let $\mathcal{U}$ be the scheme theoretic image of the multiplication $U(-\alpha_1)\times...\times U(-\alpha_n)\to G$ (the order of $U(-\alpha_i)$'s are fixed).

\begin{prop}[Finer Description of $p_1^{-1}(gP/P)$, II]
\label{finer2}
The fiber of $p_1:X_P(w_{\bullet})\to X_P(w_1)$ over the identity flag $P/P\in X_P(w_1)$ is isomorphic to $(U(J,\vec{n})\cap \mathcal{U})\cdot X_P(w_2,...,w_r)$.
\end{prop}

\begin{proof}[Proof of Proposition \ref{finer}]
The closed points of $X_P'(w_1)$ are identified with the closed points of $X_{P_{red}}'(w_1)$ via the universal homeomorphism $\pi:G/P_{red}\to G/P$. Therefore the closed point $gP/P$ of $X_P'(w_1)$ has the form $u'wP/P$ with $u'\in P_{red}(k')$. 
Since $X_P(w_2,...,w_r)$ is invariant under the left multiplication by $P_{red}$, combined with Prop.\;\ref{Fiber of First Projection}, it suffices to show that
we have the following identity of closed subschemes of $G/P_{red}$: 
\begin{equation}
\label{03}
  u'w_1P/P_{red}\cap \overline{P_{red}w_1P_{red}/P_{red}}= u'\cdot ( U(w_1(J,\vec{n}))\cap P_{red})\cdot w_1P_{red}/P_{red}.  
\end{equation}

Since $u'w_1P/P_{red}$ is an infinitesimal thickening of the closed point $u'w_1P_{red}/P_{red}$, which is in the interior $P_{red}w_1P_{red}/P_{red}$, we have that
\begin{equation}
\label{04}
     u'w_1P/P_{red}\cap \overline{P_{red}w_1P_{red}/P_{red}}
=  u'w_1P/P_{red}\cap P_{red}w_1P_{red}/P_{red}.
\end{equation}

It is then easy to see that the right hand sides of (\ref{03}) and (\ref{04}) agree.
\end{proof}

\begin{proof}[Proof of Proposition \ref{finer2}]
From above we see that it suffices to show that
\begin{equation}
\label{eq_in_finer}
  P/B\cap\overline{Bw_1B/B}=(U(J,\vec{n})\cap \mathcal{U})B/B.  
\end{equation}

Since $P/B$ is an infinitesimal thickening of the identity flag $B/B$, which is in the interior of the open and dense opposite Schubert cell $Y_{B}'(id)=U(-R^+)B/B$ as discussed in Prop.\;\ref{opposite}, we have
\begin{equation}
    \label{fiberatid}
    P/B\cap\overline{Bw_1B/B}=P/B\cap (U(-R^+)B/B\cap \overline{Bw_1B/B}).
\end{equation}

The latter intersection is the closure of the open Richardson variety $Z_{B}(w_1,v)=X_{B}'(w_1)\cap Y_{B}'(id)$ inside $Y_{B}'(id)$. We use the Deodhar decomposition to obtain a parametrization of the open Richardson variety. Let us refer to \cite[(4.6)\& Prop.\;5.2]{marsh2004parametrizations}
for a detailed discussion of Deodhar decomposition. What is useful for us is that \cite[Prop\;5.2]{marsh2004parametrizations} gives us that
\[Z_{B}(w_1,v)=x_{-\alpha_1}(k^*)\cdot ...\cdot x_{-\alpha_n}(k^*)B/B,\]
where $w_1=s_{\alpha_1}...s_{\alpha_n}$ is a reduced expression of $w_1$ in terms of simple reflections, and $k^*$ is the units in the base field $k$. 
Taking the closure in $Y_{B}'(-id)$, we have the (\ref{eq_in_finer}) as desired. 
\end{proof}

From the proof of Prop.\;\ref{finer2} above, we can see that, by embedding the infinitesimal scheme $gP/P_{red}$ into an open subscheme of $G/P_{red}$ containing the point $gP_{red}/P_{red}$, the problem of understanding the nonreducedness of the fibers of $p_1$ is related to the structure of Richardson varieties, as defined after Prop.\;\ref{opposite}, in $G/P_{red}$.

Below we use Propositions \ref{finer} and \ref{finer2} to give examples of exotic phenomena related to the morphisms among BSDH varieties with nonreduced $P$.

In Ex.\;\ref{Nonreduced Fibers of $p_1$} below, we give an example where all the fibers of $p_1$ over $k$-points are isomorphic and nonreduced. 

\begin{example}[Nonreduced Fibers of $p_1$]
\label{Nonreduced Fibers of $p_1$}
Let $G=SL_5$. Let $P_{red}=B$ be a fixed Borel subgroup. Let $\alpha$, $\beta$, $\gamma$, and $\delta$ be the four positive simple roots labeling the four nodes in the Dynkin diagram $A_4$ from left to right respectively. Let $P=U(-\beta,1)\cdot B$, i.e., $U(J,\vec{n})=U(-\beta,1)$. We consider the first projection
\[ p_1: X_P(s_{\alpha}s_{\beta},\;s_{\delta})\to X_P(s_{\alpha}s_{\beta}).\]

By Prop.\;\ref{finer}, over a general point in the Schubert cell $Bs_{\alpha}s_{\beta}P/P$, the fiber of $p_1$ is isomorphic to
\[(U(-\beta,1)\cap U(s_{\beta}s_{\alpha}(R^+)))\cdot X_P(s_{\delta})=U(-\beta,1)\cdot X_P(s_{\delta}).\]

As $U(-\beta,1)$ is not contained in the stabilizer of points of $X_P(s_{\delta})$, we have that $U(-\beta,1)\cdot X_P(s_{\delta})$ is nonreduced.

Prop.\;\ref{finer} also entails that the fiber of $p_1$ over the identity flag $P/P$ is isomorphic to
\[(U(-\beta,1)\cap U(\{-\alpha,-\beta\}))\cdot X_P(s_{\delta})=U(-\beta,1)\cdot X_P(s_{\delta}),\]
which is isomorphic to the general fiber.

We now consider the fibers over points in the Schubert cells $Bs_{\alpha}P/P$ and $Bs_{\beta}P/P$. Over $Bs_{\alpha}P$, it suffices to determine
\[s_{\alpha}P/B\cap X_B(s_{\alpha}s_{\beta})=s_{\alpha}P/B\cap s_{\alpha}U(-R^+)B/B\cap X_B(s_{\alpha}s_{\beta}).\]

By Prop.\;\ref{opposite} or \cite[Prop.\;1.3.5]{Brion}, we have that 
\[s_{\alpha}U(-R^+)B/B\cap X_B(s_{\alpha}s_{\beta})=Bs_{\alpha}B/B\times (U(-R^+)s_{\alpha}B/B\cap X_B(s_{\alpha}s_{\beta})).\]

By Deodhar decomposition \cite[(4.6)\& Prop.\;5.2]{marsh2004parametrizations}, 
the second factor is $U(-R^+)s_{\alpha}B/B\cap X_B(s_{\alpha}s_{\beta})=s_{\alpha}U(-\beta)B/B$.
Therefore we have that 
\[s_{\alpha}P/B\cap X_B(s_{\alpha}s_{\beta})= s_{\alpha}U(-\beta,1) B/B\cap s_{\alpha} U(\{-\alpha,-\beta\})B/B=s_{\alpha}U(-\beta,1)B/B.\]
Therefore we have that the fiber over a closed point in $Bs_{\alpha}P/P$ is isomorphic to $U(-\beta,1)\cdot X_P(\delta)$.

Over the Schubert cell $Bs_{\beta}P/P$, we use Prop.\;\ref{opposite} and Deodhar decomposition again to obtain that
\[s_{\beta}P/B\cap X_B(s_{\alpha}s_{\beta})=s_{\beta}U(-\beta,1)B/B\cap s_{\beta}U(\{-\beta,-\alpha-\beta\})B/B=s_{\beta}U(-\beta,1)B/B.\]

In conclusion, we see that the fibers of $p_1$ over all the $k$-points of $X_P(w_1)$ are isomorphic to an infinitesimal thickening of $X_P(\delta)$, $U(-\beta,1)\cdot X_P(\delta)$.
\end{example}

The calculation in Ex.\;\ref{Nonreduced Fibers of $p_1$} shows that, in order to determine the fiber of $p_1:X_P(w_{\bullet})\to X_P(w_1)$ over a point in the Schubert cell $BvP/P$ with $v<w$, we can always first use the identity 
\[vP/P_{red}\cap X_{P_{red}}(w_1)=vP/P_{red}\cap (vU(-R^+)P_{red}/P_{red}\cap X_{P_{red}}(w_1)).\]
The latter intersection is then isomorphic to $X_{P_{red}}'(v)\times Z$, where $Z$ is the closure of the open Richardson variety (defined after Prop.\;\ref{opposite}) $Z_{P_{red}}'(w_1,v)$ inside $Y_{P_{red}}'(v)$.
The part $X_{P_{red}}'(v)$, combined with the first half of the proof of Prop.\;\ref{finer}, entails that the fiber is an infinitesimal thickening of 
\[(U(J,\vec{n})\cap U(v^{-1}(I)))\cdot X_P(w_2,...,w_r),\]
which may already be nonreduced. 
The study of the fiber of $p_1$ is then reduced to the study of closures of open Richardson varieties in opposite Schubert cells.
From this procedure, we see that the fibers of $p_1:X_P(w_{\bullet})\to X_P(w_1)$ (or $\pi:X_{P_{red}}(w_1)\to X_P(w_1)$) have different descriptions over different parts of $X_P(w_1)$. 
Example \ref{sl52} below is an example where the fibers of $p_1: X_P(w_{\bullet})\to X_P(w_1)$ are not isomorphic to each other: 

\begin{example}[Non-isomorphic fibers of $p_1$]
\label{sl52}
Let $G$, $P_{red}$, and the roots be as in Example \ref{Nonreduced Fibers of $p_1$}.
Let $P=U(-\alpha,1)\cdot B$. We consider the first projection 
\[p_1: X_P(s_{\alpha}s_{\beta},s_{\delta})\to X_P(s_{\alpha}s_{\beta}).\]

Running the calculation as in Example \ref{Nonreduced Fibers of $p_1$} again, we have that the fibers of $p_1$ over the fixed points of $X_{P}(s_{\alpha}s_{\beta})$ are:
\begin{align*}
    p_1^{-1}(s_{\alpha}s_{\beta}P/P) &= X_P(s_{\delta}); & p_1^{-1}(s_{\alpha}P/P) & = U(-\alpha,1)\cdot X_P(s_{\delta});\\
    p_1^{-1}(s_{\beta}P/P) &= X_P(s_{\delta}); & p_1^{-1}(P/P) &= U(-\alpha,1)\cdot X_P(s_{\delta}). 
\end{align*}
\end{example}

In the Example \ref{sl52}, we have reduced general fibers and nonreduced special fibers (over $X_P(s_{\alpha})$). We then have the natural question: Can the situation be reversed? The following general lemma gives a negative answer.

\begin{lemma}
Let $X$ and $Y$ be two geometrically integral schemes over a field $k$.
Let $f:X\to Y$ be a proper dominant $k$-morphism.
Then the set $\{y\in Y| \;X_y\text{ is geometrically reduced}\}$ is open in $Y$.
\end{lemma}
\begin{proof}
We follow the argument as in \cite[0C0E,0C0D]{stacks-project}. 

Using the same argument as in \cite[0C0E]{stacks-project} we are reduced to the case where $Y=Spec(R)$ with $R$ a discrete valuation ring. It suffices to show that if the generic fiber $X_{\eta}$ is nonreduced, then the special fiber $X_{y}$ is also nonreduced.
We then imitate the argument as in \cite[\;\!0C0D]{stacks-project} to show it:

Let $x\in X_{\eta}$ be a point such that the ring $\mathcal{O}_{X_{\eta},x}$ is nonreduced. 
Let us identify $x$ with its image in $X$ via the open embedding $X_{\eta}\hookrightarrow X$. 
We then have that the ring $\mathcal{O}_{X,x}$ is also nonreduced.
Since the morphism $f$ is proper, the point $x$ has a specialization $x'$ inside the special fiber $X_y$.
We then have that the ring $\mathcal{O}_{X,x'}$ is also nonreduced.
Let $f\in \mathcal{O}_{X,x'}$ be a nilpotent element.
Let $\pi$ be the uniformizer of $R$.
By Krull's intersection theorem \cite[\;\!00IP]{stacks-project}, there exists a natural number $n$ such that $f=\pi^n f'$ and $f' \notin \pi \mathcal{O}_{X,x'}$. 
Since $X$ is integral and $f$ is dominant, we have that $\pi$ is not a zero divisor in $\mathcal{O}_{X,x'}$.
Therefore $f'$ is also a nilpotent element in $\mathcal{O}_{X,x'}$.
Since $f' \notin \pi\mathcal{O}_{X,x'}$, its image in $\mathcal{O}_{X,x'}/\pi\mathcal{O}_{X,x'}=\mathcal{O}_{X_y,x'}$ is a nonzero nilpotent element.
Thus the special fiber $X_y$ is not reduced.
\end{proof}

In the examples above, no explicit equations are used. 
In Ex.\;\ref{nonreducedfiber} below, we give another example of the nonreduced fibers of $p_1: X_P(w_{\bullet})\to X_P(w_1)$ using defining equations.

\begin{example}[Equations for Nonreduced Fibers]
\label{nonreducedfiber}
Let $G=SL_3$, and let $P_{red}=B$, the fixed Borel subgroup. Let $\alpha$ and $\beta$ be the two simple positive roots. Let $P=U(-\alpha,1)\cdot P_{red}$.
By \cite[231]{LM}, we have an embedding 
\[X_{B}(s_{\alpha},s_{\beta})\hookrightarrow Gr(1,3)\times Gr(2,3)\cong \mathbb{P}^2\times 
\check{\mathbb{P}}^2\]
given by the map $(g_1 B, g_2 B)\mapsto (g_1\langle e_1\rangle,  g_2 \langle e_1,
e_2\rangle)$, where each $e_i$ is a standard coordinate of $\mathbb{A}^3$.
The image satisfies the relation $\langle e_1, e_2\rangle \supset g_1\langle e_1\rangle \subset g_2\langle e_1, e_2\rangle$. Let the homogeneous coordinates be $(x: y: z; \;a:b:c)$, then the defining equations for the
image of $i$ are $ax+by+cz=0$ and $z=0$. 

Similarly, we have an embedding
\[X_P(s_{\alpha},s_{\beta})\hookrightarrow Frob(Gr(1,3))\times Gr(2,3)\cong \mathbb{P}^2\times \check{\mathbb{P}^2}\]
given by the map $(g_1 P, g_2P)\mapsto (Frob(g_1\langle e_1\rangle), g_2\langle e_1, e_2\rangle)$.
The image satisfies the relation $\langle e_1, e_2\rangle \supset g_1\langle e_1\rangle \subset Frob(g_2\langle e_1, e_2\rangle)$.
The defining equations are $z=0$ and $a^p x+b^p y+c^p z=0$.
Restricting the embedding above to the first factor, we have an embedding $X_P(s_{\alpha})\hookrightarrow Frob(Gr(1,3))$ with image isomorphic $\mathbb{P}^1$, defined by $z=0$.
The first projection $p_1:X_P(s_{\alpha},s_{\beta})\to X_P(s_{\alpha})$ is the first projection $\mathbb{P}^2\times\check{\mathbb{P}}^2\to \mathbb{P}^2$ restricted to $X_P(s_{\alpha},s_{\beta})$. 
For every point with coordinate $(x_0:y_0:0)\in X_P(s_{\alpha})$, the fiber of $p_1$ is the subscheme in $\check{\mathbb{P}}^2$ defined by $x_0a^p+y_0b^p+0c^p=0$.
Since the field $k$ is perfect and $char(k)=p$, we have that the fiber is defined by $(x_0^{1/p}a+y_0^{1/p}b)^p=0$, hence nonreduced.
\end{example}

The example above is also an example of non-normal BSDH varieties when $P$ is nonreduced:
\begin{example}[Non-normal BSDH Variety]
\label{nonnormalbsdh}
Take the $X_P(s_{\alpha},s_{\beta})$ as in Ex.\;\ref{nonreducedfiber}. It is the subvariety of $\mathbb{P}^2\times\check{\mathbb{P}}^2$ defined by the homogeneous ideal $\langle a^px+b^py,z\rangle$. On the chart $x \ne 0$ $c \ne 0$, the variety is defined by the ideal $\langle a^p+b^py,z\rangle$. By computing the Jacobian, we see that the singular locus consists of the points with coordinate $(a:0:c;\; x:y:0)$. Therefore the singular locus is of codimension 1.
\end{example}

\begin{example}[Non-normal Schubert Varieties]
\label{Nonnormalschubert}
For each $n\ge 2$, we can find a non-normal Schubert variety of dimension $n$:

We use the twisted incidence varieties, which are called unseparated incidence varieties in \cite[Sec.\;2.2]{Lau96} (Schubert varieties are not discussed there). Take $G=SL_{n+1}$. The parabolic $P$ is set up so that $G/P_{red}$ is the incidence variety 
$\sum_{i=1}^{n+1}x_iy_i=0$
in $\mathbb{P}^n\times \check{\mathbb{P}}^n$, and $G/P$ is the unseparated incidence variety 
$\sum_{i=1}^{n+1}z_iw_i^p=0$.
The quotient $\pi:G/P_{red}\to G/P$ is given by the ring map $z_i\mapsto x_i^p$ and $w_i\mapsto y_i$. 
For example, when $n=3$, for any $k$-algebra $A$, the $A$-points $P(A)$ consists of the matrices of the form
\[
\begin{bmatrix}
* & * & * & * \\
\epsilon & * & * & * \\
\epsilon & * & * & * \\
0 & 0 & 0 & *
\end{bmatrix},\;\;\; \epsilon^p=0.
\]

By \cite[Ex.\;1.2.3.5]{Brion}, the Schubert varieties of $G/P_{red}$ are of the form $I_{i,j}'$ with $1\le i,j\le n+1$, $i \ne j$, where each $I_{i,j}'$ is defined by the homogeneous ideal
\[x_{i+1}=...=x_{n+1}=y_1=...=y_{j-1}=0.\]

Therefore the Schubert varieties of $G/P$ are of the form $I_{i,j}$ where each $I_{i,j}$ is defined by the homegenous ideal
\[z_{i+1}=...=z_{n+1}=w_1=...=w_{j-1}=0.\]

When $n=2$, we take the Schubert variety $I_{2,1}$. The homogeneous ideal is $\langle z_1w_1^p+z_2w_2^p,z_3\rangle$, comparing the equations, we see that $I_{2,1}$ is isomorphic to the BSDH variety $X_P(s_{\alpha},s_{\beta})$ discussed in Ex.\;\ref{nonnormalbsdh}, hence non-normal.

When $n\ge 3$, we take the Schubert variety $I_{3,1}$. This is an $n$-dimensional Schubert variety. In the affine chart $z_{1} \ne 0, w_{n+1} \ne 0$, the defining ideal is
$\langle z_2w_2^p+z_3w_3^p,w_1,z_4,...,z_{n+1}\rangle$. 
By computing the Jacobian and noticing that $char(k)=p$, we have that the singular locus is defined by
$w_1=w_2^p=w_3^p=z_4=...=z_{n+1}=0$, which has codimension 1 in $I_{3,1}$. Therefore $I_{3,1}$ is non-normal.
\end{example}

M. Brion points out that the Ex.\;\ref{Nonnormalschubert} above shows that in general the natural morphism ${}^B\!X_{P_{red}}(w)\to {}^B\!X_{P}(w)$ is not flat.

\begin{rmk}[Nonflat Morphisms ${}^B\!X_{P_{red}}(w)\to {}^B\!X_P(w)$]
\label{notflatincidence}
We show that the natural morphisms $\pi: I'_{2,1}\to I_{2,1}$ and $\pi: I'_{3,1}\to I_{3,1}$, as in the Ex.\;\ref{Nonnormalschubert} above, are not flat. 

Suppose $\pi: I'_{2,1}\to I_{2,1}$ is flat, then by \cite[Cor.\;6.5.2.(i)]{EGA-IV.2}, we see that if $x$ is a regular point of $I'_{2,1}$, then $\pi(x)$ is a regular point of $I_{2,1}$. However, the singular locus of $I_{2,1}$ is of codimension 1, while $I'_{2,1}$ is normal. Hence $\pi: I'_{2,1}\to I_{2,1}$ is not flat. The same reasoning shows that $\pi: I'_{3,1}\to I_{3,1}$ is not flat either.
\end{rmk}

In contrast to Ex.\;\ref{Nonnormalschubert}, we have the following proposition:

\begin{prop}
No matter whether the parabolic $P$ is reduced or not, one dimensional Schubert varieties of $G/P$ are always isomorphic to $\mathbb{P}^1$.
\end{prop}
\begin{proof}
The argument in Ex.\;\cite[Ex.\;1.3.4.(2)]{Brion} still works here. Namely, any one dimensional Schubert variety in $G/P$ has the form ${}^B\!X_P(s)$ with $s$ a simple reflection, we have that blue${}^B\!X_P(s)\cap U(-I)P/P$, being $T$-invariant, is the affine line $\mathbb{A}^1$ in the direction $s$. Therefore it remains to check that  {${}^B\!X_P(s)$} is smooth at the other $T$-fixed point $sP/P$, which follows from the smoothness of the Schubert cell $B s P/P$.
\end{proof}

Proposition \ref{Fiber of First Projection}
implies that the first projection $X_P(w_1,...,w_r)\to X_P(w_1)$ in general cannot be a Zariski locally trivial fibration as the domain is reduced but, in general, the fibers over closed points are not. 
This failure to be a Zariski locally trivial fibration can be explained from another perspective as shown in the following Remark \ref{multiplication}:

\begin{rmk}
\label{multiplication}
Fix a $k$-algebra $A$ and we only consider $A$-points in this remark.

Let us first recall why $X_{P_{red}}(w_{\bullet})$ is a Zariski locally trivial fibration in the classical case:
Let $g_0=id\in G$.
\[V'(w_{\bullet}):=\{(g_1,...,g_r)\in G^r|\; g_{i-1}^{-1}g_i\in P_{red} w_i P_{red}\}.\]

The multiplication map gives an isomorphism
$m: V'(w_1)\times V'(w_2)\xrightarrow{\sim} V'(w_1, w_2)$,
which induces an isomorphism
$X_{P_{red}}'(w_1)\times X_{P_{red}}'(w_2)\xrightarrow{\sim} X_{P_{red}}'(w_1,w_2)$.

Now we consider $X_P(w_{\bullet})$ instead of $X_{P_{red}}(w_{\bullet})$.
We have that 
$X_P'(w_1, w_2)\cong V'(w_1, w_2)/\sim_1$ where $\sim_1$ is an equivalence relation defined as $(g_1, g_2)\sim_1(g_1 p_1,g_2 p_2)$ for some $p_i\in P$.
Note this is not a quotient by group action, as $P$ does not act on $V'(w_1,w_2)$. 
Pulling back $\sim_1$ along $m$ and taking the quotient by $P$ on the second factor, we have that $X_P'(w_1,w_2)\cong (V'(w_1)\times X_P'(w_2))/\sim_2$, where $\sim_2$ is an equivalence relation defined as
$(g_1,b)\sim_2 (g_1 p,p^{-1} b)$ for some $p\in P.$

On the other hand, we have that
$X_P(w_1)\times X_P(w_2)\cong (X_P(w_1)\times P_I)\times^{P_I} X_P(w_2)$. 
The map $V'(w_1)\to X_P(w_1)\times P_I$, defined by $uw_1p\mapsto (uw_1P, p)$ with $u\in U(I\cap w_1(-I))$, is the quotient by the relation $uw_1p\sim u\epsilon w_1 p$ for $\epsilon\in U(I\cap w(J,\vec{n}))$.
Therefore we see that 
$X_P'(w_1)\times X_P'(w_2)\cong (V'(w_1)\times X_P'(w_2))/\sim_3$, where $\sim_3$ is an equivalence relation defined as follows: 
Every $g\in V'(w_1)$ can be written as $g=u(g)w_1p(g)$ uniquely for $u(g)\in U(I\cap w(-I))$ and $p(g)\in P_I$. The equivalence relation $\sim_3$ is defined as 
$(g, b)\sim_3(gp,p^{-1} b)\sim_3 (u(g)\epsilon w_1 p(g), b)$, for $p\in P_I$ and $\epsilon\in U(I)\cap U(w(J,\vec{n}))$.

When $P \ne P_{red}$, the first part of $\sim_3$ is included in $\sim_2$ but the second part of $\sim_3$ is not, so we see that $\sim_2$ and $\sim_3$ are in general not comparable, hence the multiplication $m$ above in general does not induce a morphism between $X_P'(w_1)\times X_P'(w_2)$ and $X_P'(w_1,w_2)$.

\end{rmk}

Below we give an example of nonreduced fibers of the last projection, where $p>0$ can be any prime number.
This example \ref{lastprojection2} is also curious because the target of the last projection is the whole variety $G/P$.

\begin{example}[Nonreduced Fibers of $p_3$]
\label{lastprojection2}
We use the setup in Ex.\ref{nonreducedfiber}, i.e., $G=SL_3$ and $P=U(-\alpha,1)\cdot B$. Consider the last projection
\[p_3: X_P(s_{\beta}, s_{\alpha}, s_{\beta})\to X_P(s_{\beta}s_{\alpha}s_{\beta})=G/P.\]

A point in the big Schubert cell $X_P'(s_{\beta}s_{\alpha}s_{\beta})$ has the form $us_{\beta}s_{\alpha}s_{\beta}P/P$ for some $u\in R_u(B)=U(\alpha,\beta,\alpha+\beta)$. We have that 
\[us_{\beta}s_{\alpha}s_{\beta} P/P= us_{\beta}s_{\alpha}s_{\beta} U(-\alpha,1) P/P=uU(\beta,1) s_{\beta}s_{\alpha}s_{\beta} P/P.\] 

Therefore the fiber 
$p_3^{-1}(us_{\beta}s_{\alpha}s_{\beta}P/P)$ is isomorphic to the direct product 
\[uU(\beta,1)s_{\beta}P/P\times uU(\beta,1) s_{\beta}s_{\alpha}P/P= us_{\beta} U(-\beta,1) P/P\times us_{\beta}s_{\alpha} U(-\alpha-\beta,1) P/P.\] 

Since $-\beta, -\beta-\alpha \notin J$, we have that $p_3^{-1}(us_{\beta}s_{\alpha}s_{\beta}P/P)$ is nonreduced.
Therefore the last projection $p_3$ is not birational.
\end{example}

\section{Nice Topology}\label{top}

Although the geometry of $X_P(w_{\bullet})$ differs greatly from that of classical BSDH varieties, the topology of the newly constructed BSDH varieties remains the same as for the corresponding classical BSDH varieties, due to the fact that $\pi:G/P_{red}\to G/P$ is purely inseparable, finite, and surjective, hence a universal homeomorphism \cite[Prop.\;2.4.4]{EGA-IV.2}. 
Over a finite or algebraically closed field, we can then generalize some results in \cite{dCHL} to the case involving $X_P(w_{\bullet})$ for a nonreduced parabolic $P$. Namely, given a generalized convolution morphism $f:X_P(w_{\bullet})\to X_Q(w_{\theta,\bullet}'')$ (defined in Def. \ref{generalizedconvolution}), a special case of which is the last projection $p:X_P(w_{\bullet})\to X_P(w_{\star})$, we prove that the decomposition theorem holds for $f$, and that the push forward of the intersection complex $f_*\mathcal{IC}_{X_P(w_{\bullet})}$ is \textit{good}, which is a notion defined in \cite[Def.\;1.2.1]{dCHL} and recalled below in Def. \ref{gooddef}.\\

We first recall some notation. Let $k$ be a finite field. For $X$ a separated scheme of finite type over $k$,
the triangulated category $D^b_m(X,\overline{\mathbb{Q}}_l)$ is the mixed, bounded, and constructible "derived" category with self dual perversity as in \cite[Sec.\;2.2.10-19]{BBD}. The truncation functors for the standard t-structure are denoted $\tau_{\le i}$ and $\tau_{\ge i}$ for $i\in  \mathbb{Z}$. Let $\bar{x}\in X(\bar{k})$ be a geometric point, and let $x$ be the closed point that is the image of $\bar{x}: \text{Spec}(\bar{k})\to X$. For every $F\in D_m^b(X,\overline{\mathbb{Q}}_l)$, the stalk of the $i$-th cohomology sheaf $\mathcal{H}^i(F)_{\bar{x}}$ is a $Gal(k(x)_{s}/k(x))$-module, where $k(x)_s$ is the separable closure of $k(x)$. In the rest of the paper, when we consider the stalk of $F\in D_m^b(X,\overline{Q}_l)$ at some $\bar{x}\in X(\bar{k})$ as a Galois module, we always mean the structure of a $Gal(k(x)_s/k(x))$-module.

Now we can define the notion of \textit{good}:

\begin{definition}\label{gooddef}
We say that $F\in D^b_m(X,\overline{\mathbb{Q}}_l)$ is good if $F$ is
\begin{enumerate}
\item Semisimple, i.e., it is isomorphic to a direct sum of shifted simple perverse sheaves, which, by \cite[Sec.\;4.3]{BBD}, are of the form $j_{!*}(L[dimV])$, where $j:V\hookrightarrow X$ is an inclusion of irreducible smooth subvariety, and $L$ is an irreducible lisse $\overline{\mathbb{Q}}_l$-sheaf over $V$;
\item Frobenius semisimple, i.e., for every $\bar{x}\in X(\bar{k})$, the stalks of the cohomology sheaves $\mathcal{H}^i(F)_{\bar{x}}$ are semisimple as graded Galois modules;
\item Even, i.e., $\mathcal{H}^i(F)_{\bar{x}}$ is trivial for $i$ odd;
\item Very pure with weight zero, i.e., let $F^{\vee}$ be the Verdier dual of $F$, then for each degree $i\in  \mathbb{Z}$ and $\bar{x}\in X(\bar{k})$, both stalks $\mathcal{H}^i(F)_{\bar{x}}$ and $\mathcal{H}^i(F^{\vee})_{\bar{x}}$ have weight $i$;
\item Tate, i.e., for every $\bar{x}\in X(\bar{k})$, the stalk $\mathcal{H}^i(F)_{\bar{x}}$ is isomorphic to a direct sum of Tate modules $\overline{\mathbb{Q}}_l(-k)$ of weight $2k$ for possibly varying $k\in  \mathbb{Z}$.
\end{enumerate}
\end{definition}

The following lemma is probably well known, but we cannot find an explicit reference:

\begin{lemma}
\label{seplemma}
Given any finite purely inseparable field extension $k\hookrightarrow k'$, we have a commutative diagram of field extensions:
\begin{equation}
\label{sep}
\begin{tikzcd}
\bar{k} \arrow[hookleftarrow]{r}&
k'_s \arrow[hookleftarrow]{r}&
k_s\\
&
 k' \arrow[hookleftarrow]{r}
 \arrow[hookrightarrow]{u}&
 k\arrow[hookrightarrow]{u}
\end{tikzcd},
\end{equation}
where $k_s$ and $k_s'$ are separable closures of $k$ and $k'$, and $\bar{k}$ is an algebraic closure of both $k$ and $k'$.
By restriction, we have a group isomorphism $\psi: Gal(k_s'/k')\to Gal(k_s/k)$. 
\end{lemma}
\begin{proof}
Firstly, we can take $k_s'$ to be $k'\otimes_k k_s$ by the primitive element theorem from basic algebra. 

The extension $\bar{k}/ k_s$ is purely inseparable hence any intermediate extension of it is again purely inseparable. Therefore $k_s\hookrightarrow k_s'$ is purely inseparable.

We then have that
$Aut(k_s'/k_s)=1$, i.e.,
a field automorphism of $k_s'$ that fixes $k_s$ must be trivial, so $\psi$ is injective. 

On the other hand, any $ \sigma\in Aut(k_s)$ can be extended to an automorphism $ \sigma'\in Aut(k_s')$. If $ \sigma$ fixes $k$, then for any element $a\in k'$, we must have 
\[0= \sigma'(a^m)-a^m= \sigma'(a)^m-a^m=( \sigma'(a)-a)^m,\]
where $m$ is a power of $p$ so that $a^m\in k$. Therefore $ \sigma'$ fixes $k'$, and $ \sigma'=\psi^{-1}( \sigma)$, hence $\psi$ is bijective. 
\end{proof}

\begin{prop}
\label{good}
Let $k$ be a finite field. Let $f:X\to Y$ be a finite, surjective, and purely inseparable morphism of schemes over $k$. Let $F\in D_m^b(X,\overline{\mathbb{Q}}_l)$ be Frobenius semisimple as defined in Def. \ref{gooddef}.(2). Then $f_*F\in D_m^b(Y,\overline{\mathbb{Q}}_l)$ is again Frobenius semisimple. Moreover, the weights of $f_*F$ are the same as the weights of $F$.
\end{prop}
\begin{proof}
We first show that $f_*F$ is Frobenius semisimple. By defintion, we need to show that for every geometric point $\bar{y}\in Y(\bar{k})$, where $y$ is the image closed point of $Y$, we have that the stalk $\mathcal{H}^i(f_*F)_{\bar{y}}$ is a semisimple Galois module. 

We want to reduce to the case where $F$ can be identified with a $\overline{\mathbb{Q}}_l$-sheaf.

Since $f$ is finite, the functor $f_*$ equals the functor $Rf_*$ and is thus exact. 
Therefore, we have
\[\mathcal{H}^i(f_*F)\cong \mathcal{H}^i(f_*\tau_{\ge i}\tau_{\le i} F)\cong f_*\tau_{\ge i}\tau_{\le i} F[i]\cong f_*\mathcal{H}^i(F),\]
isomorphisms of $\overline{\mathbb{Q}}_l$-sheaves placed at degree 0.

By assumption, $F$ is Frobenius semisimple, so $\mathcal{H}^i(F)_{\bar{x}}$ is a semisimple Galois module for every geometric point $\bar{x}\in X(\bar{k})$. 

Therefore we are reduced to show that if $F$ is a $\overline{\mathbb{Q}}_l$-sheaf on $X$ so that $F_{\bar{x}}$ is Galois semisimple for every geometric point $\bar{x}\in X(\bar{k})$, then $f_*G_{\bar{y}}$ is also Galois semisimple for every geometric point $\bar{y}\in Y(\bar{k})$.

For every closed point $y$ of $Y$, let $x$ be the closed point of $X$ so that $f(x)=y$ (recall that $f:X\to Y$ is purely inseparable). Let $x'$ be the scheme theoretic fiber of $f$ over $y$. We have the following commutative diagram,
\[\begin{tikzcd}
x\arrow{r}{i_1} \arrow{rd}{\phi_1} & x'\arrow{r}{i_2} \arrow{d}{\phi} & X\arrow{d}{f}\\
& y\arrow{r}{i_3} & Y
\end{tikzcd}\]

We then have the isomorphisms of $Gal(k(y)_s/k(y))$-modules
$ (f_*F)_{\bar{y}}\cong (i_3^*f_*F)_{\bar{y}}\cong (\phi_*i_2^*F)_{\bar{y}}$ by base change. 
Now from the trivial attaching triangle associated with $i_1: x\hookrightarrow x' \hookleftarrow \emptyset: j$, we have isomorphisms of Galois modules
\[(f_*F)_{\bar{y}}\cong (\phi_*i_2^*F)_{\bar{y}}\cong (\phi_*{i_1}_*i_1^*i_2^*F)_{\bar{y}}\cong ({\phi_1}_*(i_2\circ i_1)^* F)_{\bar{y}}.\]

Since $(i_2\circ i_1)^*F$ is a $\overline{\mathbb{Q}}_l$-sheaf on $x$, and for every $\bar{x}\in x(\bar{k})$, we have that $((i_2\circ i_1)^* F)_{\bar{x}}=F_{\bar{x}}$, the pull back $(i_2\circ i_1)^*F$ is Galois semisimple.

Therefore we are reduced to the case where $X=x=\text{Spec}(k(x))$ and $Y=y=\text{Spec}(k(y))$ are two points over $k$. 

Let $n\in\mathbb{N}$. Let $L$ be a lisse sheaf of $ \mathbb{Z}/l^n  \mathbb{Z}$-modules over $x$, it is identified as a $ \mathbb{Z}/l^n \mathbb{Z}$-module, which we still denote as $L$, provided with a continuous representation of $\xi_L: Gal(k(x)_s/k(x))\to Aut_{ \mathbb{Z}/l^n \mathbb{Z}\text{-mod}}(L)$. By 
\cite[Sec.\;II.3.1.(e)]{MilneEt}, 
the $ \mathbb{Z}/l^n \mathbb{Z}$-module $f_*L$, as a representation of $Gal(k(y)_s/k(y))$, is the coinduction of the $Gal(k(x)_s/k(x))$-module $L$, via the group isomorphism $\psi: Gal(k(x)_s/k(x))\xrightarrow{\sim} Gal(k(y)_s/k(y))=:G$ defined in Lemma \ref{seplemma}:
\[f_*L\;\stackrel{\sim}{\text{\textemdash}}_{G\text{-mod}} Hom_{ \mathbb{Z}[G]}( \mathbb{Z}[Im(\psi)], \;L^{ker(\psi)}),\]
where $L^{ker(\psi)}$ denotes the invariant part of the $ \mathbb{Z}/l^n \mathbb{Z}$-module under the action of $ker(\psi)\subset Gal(k(x)_s/k(x))$. 

As $\psi$ is a group isomorphism, we see that $f_*L\cong L$ as $ \mathbb{Z}/l^n  \mathbb{Z}$-modules and the representation $Gal(k(y)_s/k(y))\to Aut_{ \mathbb{Z}/l^n \mathbb{Z}\text{-mod}}(f_*L)\cong Aut_{ \mathbb{Z}/l^n \mathbb{Z}\text{-mod}}(L)$ factors as
\[ Gal(k(y)_s/k(y))\xrightarrow{\psi^{-1}}Gal(k(x)_s/k(x))\xrightarrow{\xi_L} Aut_{ \mathbb{Z}/l^n \mathbb{Z}\text{-mod}}(L).\]

As $\xi\circ \psi^{-1}$ and $\xi$ has the same image, we see that the proposition "if $L$ is a semisimple $Gal(k(x)_s/k(x))$-module, then $f_*L$ is a semisimple $Gal(k(y)_s/k(y))$-module" is true for every $L$, a lisse sheaf $ \mathbb{Z}/l^n \mathbb{Z}$-modules on $x$, and every $n\in  \mathbb{Z}_{>1}$. 
Taking the projective limit with respect to $n\in  \mathbb{Z}_{>1}$, we see that the same proposition is true if $L$ is a lisse $ \mathbb{Z}_l$-sheaf, as such a sheaf corresponds to the projective limit of the representations of the Galois group on $ \mathbb{Z}/l^n \mathbb{Z}$ for $n\in \mathbb{Z}_{>1}$. 
Tensoring with finite fields extensions of $\mathbb{Q}_l$ and take the direct limit, we see that the same proposition is true if $L$ is a lisse $\overline{\mathbb{Q}}_l$-sheaf. Finally, we can take $L$ to be $G$ above and the proof that $f_*F\in D_m^b(Y,\overline{\mathbb{Q}}_l)$ is Frobenius semisimple is finished.\\

We now show that $f_*F$ has the same weights as $F$:

As $k(x)$ and $k(y)$ are both perfect, the purely inseparable morphism $f:x\to y$ is induced by a field automorphism $f^{\#}:k(y)\xrightarrow{\sim}k(x)$. As Frobenius commutes with field automorphisms, we see that $\psi^{-1}:Gal(k(y)_s/k(y))\to Gal(k(x)_s/k(x))$ maps the Frobenius to Frobenius. 

Alternatively, the fact that they have the same weights also follows from the general theory that $Rg_*$ preserves weights if $g$ is a proper morphism \cite[Sec.\;5.1.14]{BBD}.
\end{proof}

Our final goal is to prove the decomposition theorem and the goodness for $f_*\mathcal{IC}_{X_P(w_{\bullet})}$, where $\mathcal{IC}_{X_P(w_{\bullet})}$ is the shifted intersection complex on a BSDH variety $X_P(w_{\bullet})$, and $f$ is a generalized convolution morphism as defined in \cite[Sec.\;4.5]{dCHL}. 

We recall the definition of generalized convolution morphisms below. For the readers who want to skip the definition, it is useful to know two families of examples of generalized convolutions morphisms. The first family of examples is the $i$-th projection for $i\le r$, $p_i: X_P(w_1,...,w_r)\to X_P(w_1\star...\star w_i)$. The second family of examples is the restriction of the natural morphism $(G/P)^r\to (G/Q)^r$, for some parabolics $P\subset Q$, to $X_P(w_{\bullet})$.\\

To define the generalized convolution morphisms, we need some more notation.

Let $P\subset Q$ be two nonreduced parabolics of $G$ containing a common Borel subgroup $B$. Let $H\subset I\subset R^+$ be such that $R_u(P_{red})=U(I)\supset U(J)=R_u(Q_{red})$.
The Weyl group of the Levi factor of $P_{red}$, $W^I$, is a subgroup of that of $Q_{red}$, $W^H$. As $W_I$ is defined as the set of longest representatives of the elements in $W^I\backslash W/W^I$, we have a natural map induced by the double quotient $w\mapsto w'': W_I\to W_H$.

Let $\gamma: G/P_{red}\to G/Q_{red}$, $w\mapsto w''$ is defined so that for $w\in W_I$, we have an equality of schemes
\[Q_{red}\cdot \gamma (P_{red}w P_{red}/P_{red})= Q_{red} w'' Q_{red}/Q_{red}.\]

Let $\Upsilon: G/P\to G/Q$ be the natural morphism. Then we have the equality
\[Q_{red}\cdot \Upsilon(X_P(w))=X_Q(w'').\]

Let $\theta$ be the data of numbers $i_1,...,i_m=r'$ with $1\le i_1<...<i_m=r'\le r$. Let $w_1,...,w_r\in W_I$ and $w_1'', ...,w_r''\in W_H$. Let $i_0=0$. Define
\[w_{\theta,k}:= w_{i_{k-1}+1}\star ...\star w_{i_k},\;\;\;\;\;\; w_{\theta,k}'':= w_{i_{k-1}+1}''\star ...\star w_{i_k}''.\]

\begin{definition}
\label{generalizedconvolution}
The generalized convolution morphism $f: X_P(w_{\bullet})\to X_Q(w_{\theta,\bullet}'')$ is defined as $f:=\prod_{j=1}^m \Upsilon\circ p_{i_j}$, where $p_{i_j}$ is the $i_j$-th projection of $(G/P)^r\to G/P$ restricted to $X_P(w_{\bullet})$. Equivalently, for any $k$-algebra $A$ and $(a_1,...,a_r)\in X_P(w_{\bullet})(A)$, we define 
\[f(a_1,...,a_r):=(\Upsilon(a_{i_1}),...,\Upsilon(a_{i_m})).\]
\end{definition}

We have the following commutative diagram, where the vertical morphisms are purely inseparable and $f$ equals to the composition of two of the bottom sides:

\begin{equation}\label{cube}
    \begin{tikzcd}[row sep=1.5em, column sep = 1.5em]
    X_{P_{red}}(w_{\bullet}) \arrow[rr] \arrow[dr, swap] \arrow[dd,swap] &&
    X_{P_{red}}(w_{\theta,\bullet}) \arrow[dd] \arrow[dr] \\
    & X_{Q_{red}}(w_{\bullet}'') \arrow[rr] \arrow[dd]&&
    X_{Q_{red}}(w_{\theta,\bullet}'') \arrow[dd] \\
    X_{P}(w_{\bullet}) \arrow[rr] \arrow[dr] && X_{P}(w_{\theta,\bullet}) \arrow[dr] \\
    & X_{Q}(w_{\bullet}'') \arrow[rr] && 
    X_{Q}(w_{\theta,\bullet}'')
    \end{tikzcd}
\end{equation}

We need one more piece of notation to state our final results.
Let $X$ be a variety over a finite field. By $\mathcal{IC}_X\in D^b_m(X, \overline{\mathbb{Q}}_l)$, we mean the intersection complex starting from degree zero, e.g., if $X$ is smooth, then $\mathcal{IC}_X$ is the constant $\overline{\mathbb{Q}}_l$ placed at degree 0. 

In the statement of the theorem below, the requirement $Q_{red}\cdot \overline{X_P(w_i)}=\overline{X_P(w_i)}$ is equivalent to say $w_i$ is of $Q$-type as defined in \cite[Def.\;3.10.3]{dCHL}. It is equivalent to require $\overline{Q_{red}w_iP_{red}}=\overline{P_{red}w_iP_{red}}$. It guarantees that the generalized convolution morphism in the theorem below is surjective.

\begin{thm}[Decomposition Theorem]
\label{decomposition}
Let the base field $k$ be algebraically closed or finite. Let $f: X_P(w_{\bullet})\to X_Q(w_{\theta,\bullet}'')$ be the generalized convolution morphism in Def. \ref{generalizedconvolution}. 

If for each $i=1,...,r$, we have that $Q_{red}\cdot \overline{X_P(w_i)}=\overline{X_P(w_i)}$, then we have the decomposition theorem for $f$: 
\[f_* \mathcal{IC}_{X_P(w_{\bullet})}\cong \bigoplus_{\mathcal{O}} \mathcal{IC}_{\mathcal{O}}\otimes (\bigoplus_{j=0}^{codim \mathcal{O}} \overline{\mathbb{Q}}_l^{m_{\mathcal{O},2j}}(-j)[-2j]),\]
where $\mathcal{O}$ belongs to a finite collection of geometrically integral $Q_{red}$-invariant closed subvarieties. 

Furthermore, Poincare-Verdier duality and Relative Hard Lefschetz theorem imply the following:
 \begin{enumerate}
\item $m_{\mathcal{O}, 2j}=m_{\mathcal{O}, 2codim\mathcal{O}-2j}$;
\item $m_{\mathcal{O}, 2j}\le m_{\mathcal{O}, 2j+2}$,  for $2j< codim \mathcal{O}$.
\end{enumerate}
\end{thm}
\begin{proof}
This is true when $P$ and $Q$ are reduced by \cite[Thm.\;2.2.7]{dCHL}.

Consider the northwest-southeast diagonal slice of the cube diagram (\ref{cube}):
\begin{equation}\label{last}
\begin{tikzcd}
X_{P_{red}}(w_{\bullet}) \arrow{r}{f} \arrow[swap]{d}{\pi} & X_{Q_{red}}(w_{\theta,\bullet}'') \arrow{d}{\pi} \\
X_P(w_{\bullet}) \arrow{r}{f}& X_Q(w_{\theta,\bullet}''),
\end{tikzcd}
\end{equation}
where we abuse the language a bit for the arrows denoted $f$ and $\pi$. 
Both vertical arrows denoted by $\pi$ are finite, surjective, and purely inseparable. 

Let $\mathcal{O}'$ be a stratum that appears in the decomposition theorem for $f:X_{P_{red}}(w_{\bullet})\to X_{Q_{red}}(w_{\theta,\bullet}'')$. 

Let $\mathcal{O}:= f(\mathcal{O}')$. 

Let $j:U\hookrightarrow \mathcal{O}'$ be an open dense smooth subscheme so that $\mathcal{IC}_{\mathcal{O}'}= \mathcal{IC}_{\mathcal{O}'}((\overline{\mathbb{Q}}_l)_U)$. 

Let $j$ also denote the open embedding $j:\pi(U)\hookrightarrow \mathcal{O}$ by abuse of language. We then have that $\pi_*j_*=j_*\pi_*$. Moreover, as $\pi$ is finite, the pushforward $\pi_*$ is exact, and we have that $\pi_*$ commutes with the truncation functors $\tau^{\le i}$, $i\in \mathbb{Z}$, for the standard t-structure. Therefore, using the description of intermediate extension as iterated $j_*$'s and $\tau^{\le i}$'s we have the isomorphism:
\[\pi_*\mathcal{IC}_{\mathcal{O}'}\cong \mathcal{IC}_{\mathcal{O}}(\pi_*(\overline{\mathbb{Q}}_l)_U).\]

Notice that this isomorphism is also given by (17) of \cite[Lem.\;2.4.1]{toricfinite}.
As $\pi$ is a universal homeomorphism, we have that $\pi^*$ induces an equivalence of the category of \'{e}tale covers of $U$
and that of $\pi(U)$, hence $\pi_*(\overline{\mathbb{Q}}_l)_U=(\overline{\mathbb{Q}}_l)_{\pi(U)}$ and 
\begin{equation}
\label{equality}
    \pi_*\mathcal{IC}_{\mathcal{O}'}\cong \mathcal{IC}_{\mathcal{O}}.
\end{equation}

Therefore we have the decomposition
\begin{align*}
    f_* \mathcal{IC}_{X_P(w_{\bullet})}& \cong f_*\pi_* \mathcal{IC}_{X_{P_{red}}(w_{\bullet})}\cong \pi_*f_*\mathcal{IC}_{X_{P_{red}}(w_{\bullet})}\\
    &\cong \bigoplus_{\mathcal{O'}} \pi_*\mathcal{IC}_{\mathcal{O'}}\otimes (\bigoplus_{j=0}^{codim \mathcal{O'}} \overline{\mathbb{Q}}_l^{m_{\mathcal{O'},2j}}(-j)[-2j])\\
    &\cong \bigoplus_{\mathcal{O}} \mathcal{IC}_{\mathcal{O}}\otimes (\bigoplus_{j=0}^{codim \mathcal{O}} \overline{\mathbb{Q}}_l^{m_{\mathcal{O},2j}}(-j)[-2j]).
\end{align*}

From this we also see that the (in-)equalities (1) and (2) follow from the corresponding (in-)equalities in the case where $P$ and $Q$ are reduced.
\end{proof}

\begin{thm}[Goodness]
\label{goodthm}
Let the base field be finite. Let $f:X_P(w_{\bullet})\to X_Q(w_{\theta,\bullet}'')$ be a generalized convolution morphism as above.
Then both $\mathcal{IC}_{X_P(w_{\bullet})}$ and $f_*\mathcal{IC}_{X_P(w_{\bullet})}$ are good as in Def.\;\ref{gooddef}.
\end{thm}
\begin{proof}
In the proof of Thm. \ref{decomposition}, it is shown that the pushforward of an intersection complex by $\pi_*$ is still an intersection complex. Therefore, the pushforward of a semisimple complex by $\pi_*$ is still semisimple. 

It is easy to see that $\pi_*$ preserves evenness.
By Prop. \ref{good}, we have that $\pi_*$ preserves Frobenius semisimplicity and weights. 
As $\pi$ is a universal homeomorphism and preserves weights, we also have that $\pi_*$ preserves Tateness.

Therefore the pushforward $\pi_*$ preserves goodness.

By  \cite[Th.\;2.2.1]{dCHL}, we have that $\mathcal{IC}_{X_{P_{red}}(w_{\bullet})}$ is good. 
Apply the equation (\ref{equality}) above to $\mathcal{O}'=X_{P_{red}}(w_{\bullet})$, we have that
\[\pi_* \mathcal{IC}_{X_{P_{red}}}(w_{\bullet})\cong \mathcal{IC}_{\pi(X_{P_{red}}(w_{\bullet}))}\cong \mathcal{IC}_{X_P(w_{\bullet})}\]
is good.

By \cite[Th.\;2.2.2]{dCHL}, we have that $f_*\mathcal{IC}_{X_{P_{red}}}(w_{\bullet})$ is good. 
Therefore 
\[f_*\mathcal{IC}_{X_P(w_{\bullet})}\cong  f_*\pi_*\mathcal{IC}_{X_{P_{red}}(w_{\bullet})}=\pi_*f_*\mathcal{IC}_{X_{P_{red}}(w_{\bullet})}\]
is good.
\end{proof}

\section{Appendix: $k$-Functors, fppf-Sheaves, and Monomorphism of Schemes}
\label{appendix}
A $k$-functor is a functor from the category of $k$-algebras to the category of sets.

By taking the functor of points, we have an equivalence between the category of schemes and a full subcategory of $k$-functors that we describe below.

A $k$-functor $X$ is called local if it is a sheaf for the Zariski site structure on the category of $k$-algebras, i.e., for every $k$-algebra $A$ and every finite set $f_1,...,f_n\in A$ so that $\sum_i f_i A=A$, the following sequence is exact:
\[ X(A)\to \prod_i X(A_{f_i})\rightrightarrows \prod_{i,j} X(A_{f_i f_j}).\]

There is a natural way to define an open sub-$k$-functor of a $k$-functor \cite[Sec. I.1.7]{Jantzen}. 

By taking the functor of points, we have an equivalence between the category of schemes and the category of local $k$-functors admitting open coverings of open subfunctors which are functor of points of affine schemes \cite[Sec. I.1.11]{Jantzen}. 
From now on we do not distinguish a scheme and its associated functor of points.

There is an fppf Grothendieck topology on the category of $k$-algebras \cite[Sec. I.5.2]{Jantzen}. 

A $k$-space $X$ is a $k$-functor that is also an fppf sheaf, i.e., for every $k$-algebra $A$ and an fppf open covering  $A_1,...,A_n$ of $A$, the following sequence is exact:
\[X(A) \to \prod_i X(A_i) \rightrightarrows \prod_{i,j} X(A_i\otimes_A A_j).\]
We use the terms $k$-space and big fppf sheaf interchangeably.

Every $k$-functor $X$ has a unique fppf sheafification, which is called the associated $k$-space $\tilde{X}$. The functor $i:X\mapsto \tilde{X}$ is the left adjoint to the inclusion functor $\{k$-spaces$\}\hookrightarrow \{k-$functors$\}$ \cite[Sec. I.5.4]{Jantzen}.

Every scheme is a $k$-space \cite[Sec. III.1.1.3]{demazure1970groupes}. Every $k$-space is a local $k$-functor. Not all $k$-spaces are schemes, e.g., all algebraic spaces are $k$-spaces.

Given a morphism of functors $f:X\to Y$ between two $k$-spaces, the image $k$-functor of $f$ is  $A\mapsto f(A)(X(A))$. The image $k$-space $Im(f)$ of $f$ is the associated $k$-space of the image $k$-functor of $f$. The image $k$-space has the universal property that it is the smallest subfunctor of $Y$, which is also a $k$-space, that $f$ factors through.
In \cite[Sec. I.5.4.4]{Jantzen}, it is shown that 
\[Im(f)(A)=\bigcup_B f(B)(X(B))\cap Y(A),\]
where $B$ is taken over all fppf-$A$-algebras and we have used the inclusions $f(B)(X(B))\subset Y(B)\supset Y(A)$.

\begin{lemma}
Let $f:X\to Y$ be a flat morphism of finite presentation between schemes. If $f$ maps the space underlying $X$ surjectively onto the space underlying the scheme theoretic image $f(X)$. Then the image $k$-space is representable by the scheme $f(X)$.
\end{lemma}
\begin{proof}
Without loss of generality, we can assume $f$ is surjective. 
From the description above, we need to show that every $A$-point of $f(X)$ is in some $f(B)(X(B))$ for some fppf-$A$-algebra $B$.
For this we can assume both $X$ and $Y$ are affine. 
Then the lemma follows from the properties of $f$ that we impose.
\end{proof}

\begin{lemma}\label{image}
Let $f:X\to Y$ be a morphism between two schemes. Let $Im(f)$ denote the image $k$-space of $f$. Then $Im(f)$ is a subfunctor of $Y$. If $Im(f)$ is also a scheme, then the inclusion $i:Im(f)\to Y$ is a monomorphism of schemes that $f$ factors through. 
Furthermore, we have that $Im(f)$ is the smallest scheme with such a property, i.e., given any monomorphism between schemes $b: Z\to Y$ so that $f:X\to Y$ factors as $X\xrightarrow{a} Z\xrightarrow{b} Y$ for some morphism $a$, we can factor $X\xrightarrow{a}Z$ further as $X\xrightarrow{f_1} Im(f)\xrightarrow{c} Z$ with $c$ a monomorphism between schemes so that $c\circ f_1=a$ and $b\circ c=i$:
\[\begin{tikzcd}
Z \arrow{r}{b}& Y\\
X \arrow{r}{f_1} \arrow{u}{a} &
Im(f)\arrow{u}{i} \arrow{ul}{c}
\end{tikzcd}\]
\end{lemma}
\begin{proof}
The image $k$-functor of $f$ always satisfies the assumption  \cite[Sec. I.5.3.2]{Jantzen}.
 Therefore the image $k$-space of $f$, $Im(f)$, is a subfunctor of the scheme $Y$ that is also a $k$-space. Therefore we have that $Im(f)\xrightarrow{i} Y$ is a monomorphism between schemes. 

Given $X\xrightarrow{a}Z\xrightarrow{b} Y$ as above. As $Z\xrightarrow{b} Y$ is a mono between schemes, we have that 
\[-\circ b: Hom(\text{Spec}(A),Z)\hookrightarrow Hom(\text{Spec}(A), Y)\]
is an injection for any $k$-algebra $A$, i.e., $b(A):Z(A)\hookrightarrow Y(A)$ gives an inclusion of sets. Therefore, up to a unique isomorphism, we have that $b:Z\to Y$ is an inclusion of $k$-spaces that $f$ factors through. By the universal property of image $k$-space, we have an inclusion of $k$-spaces $c:Im(f)\to Z$ so that $b\circ c=i$. As both $Im(f)$ and $Z$ are schemes, we have that $c$ is a monomorphism of schemes. 
Finally, because that $b\circ (c\circ f_1)=b\circ a=f$ and that $b$ is a mono, we have that $a=c\circ f_1$.
\end{proof}

\begin{cor}
If $f:X\to Y$ is a monomorphism between schemes. Then $f$ is an inclusion of $k$-functors and the image $k$-space $Im(f)$ of $f$ is $X$.
\end{cor}

\begin{cor}
The image $k$-space is in general not a scheme. 
\end{cor}
\begin{proof}
In general, there does not exist a smallest monomorphism that a morphism of schemes $f:X\to Y$. A counter example is the blow down map restricted to one chart $f:\mathbb{A}^2\to \mathbb{A}^2$, $f(x,y):=(x,xy)$: see the math overflow post by R. van Dobben de Bruyn at mathoverflow.net/questions/19871/images-and-monomorphisms-of-schemes.
\end{proof}

\begin{cor}
\label{mono}
Let $f:X\to Y$ be a morphism between two schemes that is locally of finite type. Let $Im(f)$ denote the image $k$-space of $f$. Let $f(X)$ be the scheme theoretic image of $f$. If $Im(f)$ is a scheme, then we have a monomorphism between schemes $c: Im(f)\to f(X)$. In particular, the morphism $c$ is universally injective and for any $y\in f(X)$, the fiber $c^{-1}(y)$ is either empty or Spec$(k(y))$.
\end{cor}
\begin{proof}
By Lemma \ref{image}, we have the existence of such a mono $c$. The two properties of $c$ follows from \cite[Prop. 17.2.6]{EGA-IV.4}, where it is shown that a morphism locally of finite type is a mono, iff it is radicial and formally unramified, iff the fibers of morphism satisfies the condition above.
\end{proof}

\begin{example}
Let $Y$ be a nodal cubic curve with the node $o$. Let $\tilde{f}:\tilde{X}\to Y$ be the normalization and let $a_1$ and $a_2$ be the two points lying above $o$. Let $X=\tilde{X}\setminus \{a_1\}$. Let $f:X\to Y$ be the restriction of $\tilde{f}$ to $X$. Then the scheme theroetic image $f(X)=Y$, while the image $k$-space $Im(f)=X$, as $f$ is a monomorphism.
\end{example}


\renewcommand\thesubsection{\arabic{subsection}}




\end{document}